\documentclass[11pt]{amsart}

\AtBeginDocument{%
   \def\MR#1{}
}

\usepackage{amsmath}
\usepackage{amssymb}
\usepackage{graphicx,pinlabel}
\usepackage{times}
\usepackage{amscd}
\usepackage[all]{xy}

\usepackage[dvipsnames]{xcolor}
\SelectTips{cm}{}

\usepackage{hyperref}
\hypersetup{colorlinks=true,citecolor=brown,linkcolor=brown,urlcolor=black,filecolor=black}

\numberwithin{equation}{section}
\allowdisplaybreaks

\newtheorem{theorem}{Theorem}[section]
\newtheorem{proposition}[theorem]{Proposition}
\newtheorem{lemma}[theorem]{Lemma}
\newtheorem{cor}[theorem]{Corollary}
\theoremstyle{definition}

\newtheorem{example}[theorem]{Example}
\newtheorem{remark}[theorem]{Remark}

\author{Gw\'ena\"el Massuyeau}
\address{IRMA,  Universit\'e de Strasbourg \& CNRS,   67084 Strasbourg, France
\newline  \& \ { {IMB}, Universit\'e  Bourgogne Franche-Comt\'e \& CNRS, 21000 Dijon, France  }}
\email{{gwenael.massuyeau@u-bourgogne.fr}}

\author{Takuya Sakasai}
\address{Graduate School of Mathematical Sciences, 
Univerisity of Tokyo, 
3-8-1 Komaba, 
Meguro-ku, Tokyo, 153-8914, Japan}
\email{sakasai@ms.u-tokyo.ac.jp}

\subjclass[2000]{Primary~20F38, Secondary~20F14; 20F28; 20F34; 57M05; 57M27; 57N10; 57N70}
\keywords{Mapping class group;   automorphism group of a free group; homology cobordism;   Johnson homomorphism; Magnus representation; Morita trace.}

\newcommand{\Z}{\ensuremath{\mathbb{Z}}}
\newcommand{\Q}{\ensuremath{\mathbb{Q}}}

\newcommand{\Ker}{\mathop{\mathrm{Ker}}\nolimits}
\newcommand{\Hom}{\mathop{\mathrm{Hom}}\nolimits}
\newcommand{\Id}{\mathop{\mathrm{id}}\nolimits}
\newcommand{\Aut}{\mathop{\mathrm{Aut}}\nolimits}
\newcommand{\IAut}{\mathop{\mathrm{IAut}}\nolimits}
\newcommand{\End}{\mathop{\mathrm{End}}\nolimits}
\newcommand{\Der}{\mathop{\mathrm{Der}}\nolimits}
\newcommand{\Out}{\mathop{\mathrm{Out}}\nolimits}

\newcommand{\Gr}{\mathop{\mathrm{Gr}}\nolimits}

\newcommand{\Sp}{\operatorname{Sp}}
\newcommand{\trace}{\operatorname{tr}}
\newcommand{\GL}{\operatorname{GL}}
\newcommand{\SL}{\operatorname{SL}}

\newcommand{\A}{\mathcal{A}}
\newcommand{\AutF}{\mathcal{A}}
\newcommand{\IAutF}{\mathcal{I\!A}}
\newcommand{\mcg}{\mathcal{M}}
\newcommand{\Torelli}{\mathcal{I}}
\newcommand{\cob}{\mathcal{C}}
\newcommand{\cyl}{\mathcal{IC}}
\newcommand{\Hcob}{\mathcal{H}}
\newcommand{\Hcyl}{\mathcal{IH}}

\newcommand{\Abel}{\mathsf{Ab}}
\newcommand{\IAbel}{\mathsf{IAb}}
\newcommand{\Magnus}{\mathsf{Mag}}
\newcommand{\IMagnus}{\mathsf{Mag}}
\newcommand{\Trace}{\mathsf{Tr}}
\newcommand{\ITrace}{\mathsf{ITr}}
\newcommand{\ldet}{\ell\mathsf{det}}

\newcommand{\Malcev}{\mathfrak{M}}
\newcommand{\freeLie}{\mathfrak{L}}

\newcommand{\by}[1]{\stackrel{\eqref{#1}}{=}}
\newcommand{\up}{\vspace{-0.6cm}}

\title[Morita's trace maps on the group of homology cobordisms]
{Morita's trace maps on the group of homology cobordisms}

\date{September 18, 2018}

\begin{document}

\begin{abstract}
Morita introduced in 2008 a $1$-cocycle  on the group of homology cobordisms 
 of surfaces  with values in an infinite-dimensional vector space. 
His  $1$-cocycle  contains all the ``traces'' of Johnson homomorphisms 
which he introduced fifteen years earlier in his study of the mapping class group. 
In this paper, we propose a new version of Morita's $1$-cocycle based on a simple and explicit construction. 
Our $1$-cocycle is proved to satisfy several fundamental properties, 
including a  connection with the Magnus representation and the LMO homomorphism.
As an application, we show that the rational abelianization of the  group of homology cobordisms is non-trivial.
Besides, we  apply some of our algebraic methods to compare two natural filtrations on the automorphism group of a finitely-generated free group. 
\end{abstract}

\maketitle

\setcounter{tocdepth}{1}
\tableofcontents

%----------Introduction--------------

\section{Introduction}\label{sec:intro}

Let $\Sigma$ be a compact connected oriented surface of genus $g$ with one boundary component.
The {\it mapping class group}  $\mcg:= \mcg(\Sigma)$ of $\Sigma$ is the group of isotopy classes of self-homeomorphisms of $\Sigma$ 
which fix the boundary  $\partial \Sigma$ pointwisely. 
A~fruitful approach of the mapping class group, which has been developed by Johnson \cite{johnson_survey} and Morita \cite{morita}, 
consists in studying the group $\mcg$ through its action on the fundamental group $\pi:=\pi_1(\Sigma,\star)$ of the surface $\Sigma$ based at a point $\star \in \partial \Sigma$. 
Specifically, this approach is based on the \emph{Johnson filtration} which is a decreasing sequence of subgroups of the mapping class group:
$$
\mcg = \mcg[0] \supset  \mcg[1] \supset  \cdots \supset    \mcg[k] \supset \mcg[k+1] \supset \cdots
$$
For every  $k\geq 1$,  the group $\mcg[k]$ consists of 
(the isotopy classes of) the self-homeomor\-phisms of 
$\Sigma$ acting trivially on the $k$-th nilpotent quotient  of $\pi$. 
For instance, $\Torelli:= \mcg[1]$ is the subgroup of $\mcg$ acting trivially 
on the homology of~$\Sigma$, which is known as the \emph{Torelli group} of~$\Sigma$.
For every  $k\geq 1$, there is a  group homomorphism 
$\tau_k:\mcg[k] \to \mathfrak{h}_k^\Z$ with kernel $\mcg[k+1]$, 
which encodes the action of $\mcg[k]$ on the $(k+1)$-st nilpotent quotient  of $\pi$ and is referred 
to as  the  {\it $k$-th Johnson homomorphism}.
Here $\mathfrak{h}_k^\Z$ is a  free abelian group which consists of  certain tensors of degree 
$(k+2)$ in $H^\Z:=H_1(\Sigma;\Z)$, 
and a fundamental problem is then to compute the image of $\tau_k$. 
The first result in this direction is due to Morita: he showed that a certain ``trace'' map 
$\ITrace_k:\mathfrak{h}_k^\Z \to S^k( H^\Z)$, 
with values in the degree $k$ part of the symmetric algebra $S( H^\Z )$, vanishes on the image of $\tau_k$ 
and is rationally surjective for every odd $k\geq 3$.

The mapping class group $\mcg$ can be regarded as an object  of study of $3$-dimen\-sional topology 
via its embedding into the group of homology cobordisms.
A~\emph{cobordism} from $\Sigma$ to itself is a compact connected oriented $3$-manifold $M$ coming with two maps $m_+: \Sigma \to M$ and $m_-: \Sigma \to M$, 
such that  the oriented surface $\partial M$ decomposes into $-m_-(\Sigma)$, $m_+(\Sigma)$ 
and a copy of the annulus $\partial \Sigma \times [-1,1]$ 
connecting $m_-(\partial \Sigma)$ to $m_+(\partial \Sigma)$,
which we call the \emph{vertical boundary}. 
Two cobordisms $M,M'$ are considered to be \emph{equivalent} 
if there is an orientation-preserving  homeomorphism $h:M\to M'$ 
whose restriction to the boundary satisfies $h\circ m_\pm=m'_\pm$ 
and corresponds to the identity of $\partial \Sigma \times [-1,1]$ on the vertical boundary. 
A~cobordism $M$  is called a \emph{homology cobordism} if  
the boundary-parametrizations $m_+$ and $m_-$ induce isomorphisms in integral homology. 
The set $\cob:=\cob(\Sigma)$ of (equivalence classes~of) 
homology cobordisms from $\Sigma$ to itself is a monoid: 
the multiplication law is defined by gluing cobordisms in the usual way, and the unit element is the trivial cobordism $\Sigma \times [-1,1]$ with the obvious boundary parametrizations. 
Furthermore, two homology cobordisms $M,M' \in \cob$ are said to be \emph{homology cobordant} 
if they are in turn related by a $4$-dimensional oriented homology cobordism.
The quotient $\Hcob:= \Hcob(\Sigma)$ of the monoid $\cob$ by this equivalence relation
is actually a group and is called the  {\it group of homology cobordisms}  of $\Sigma$.
In fact, there are two versions of this group, namely $\Hcob_{\operatorname{smooth}}$ and  $\Hcob_{\operatorname{top}}$,
depending on whether we are considering smooth or topological $4$-manifolds.
For instance, when $g=0$, the group $\Hcob$ is naturally identified with the ``homology cobordism group'' of oriented integral homology 3-spheres. 
The topological version of this group is trivial  by a result of Freedman \cite{Freedman}.
But,  in the smooth case,  this group  denoted by $\Theta^3$  is known to be highly non-trivial by works of Fintushel and Stern \cite{fs} and Furuta \cite{furuta}.
Thus, the group $\Hcob_{\operatorname{smooth}}$ can be viewed as a higher genus version of $\Theta^3$.
In the sequel, since our considerations apply to $\Hcob_{\operatorname{smooth}}$ as well as  $\Hcob_{\operatorname{top}}$,
 we will simply use the notation $\Hcob$. 

The ``mapping cylinder'' construction gives a group homomorphism $\mcg \to \Hcob$, which turns out to be injective:
thus one can view $\mcg$ as a subgroup of~$\Hcob$. 
This viewpoint on the mapping class group, which originates from the works of Habiro \cite{habiro} and Garoufalidis--Levine \cite{gl} on finite-type invariants of $3$-manifolds,
gave new perspectives on the Johnson--Morita approach of the mapping class group.
Indeed, the Johnson filtration can be defined for the group of homology cobordisms~too:
$$
\Hcob = \Hcob[0] \supset  \Hcob[1] \supset  \cdots \supset    \Hcob[k] \supset \Hcob[k+1]  \supset \cdots.
$$
In particular, $\Hcyl:= \Hcob[1]$ consists of the classes $M\in \Hcob$ 
that have the same homology type as $\Sigma\times [-1,1]$:
hence, it is called  the \emph{group of homology cylinders}. 
For every $k\geq 1$, the $k$-th Johnson homomorphism extends to a group homomorphism $\tau_k: \Hcob[k] \to \mathfrak{h}_k^\Z$ with kernel $\Hcob[k+1]$, 
which is shown to be surjective in \cite{gl,habegger}.
Thus, the problem of determining the image of $\tau_k: \mcg[k] \to \mathfrak{h}_k^\Z$ 
is deeply related to the recognition of elements of $\mcg$ inside $\Hcob$.

In this spirit, Morita has considered in \cite{morita_GD} a generalization of his ``trace'' maps 
to the group of homology cobordisms. 
Specifically, he produced a single group $1$-cocycle on $\Hcob$  whose restriction to $\Hcob[1]=\Hcyl$ contains $\tau_1$
and whose restriction to $\Hcob[k]$ contains the composition $\ITrace_k \circ \tau_k$  for every odd $k\geq 3$.
This group $1$-cocycle takes values in an infinite-dimensional $\Q$-vector space 
in which it has a ``Zariski-dense'' image.
As an application, he deduced that the abelianization of the group $\Hcyl$ has an infinite rank
and he was able to construct some (possibly non-trivial) cohomology classes in $\Hcob$.
Morita's construction, which uses  some elements of the theory of algebraic groups, 
is somehow intricate in that it involves several isomorphisms of Lie algebras which are not always explicit.

In this paper, we propose a new version of Morita's ``generalized trace'' on the group $\Hcob$.
Although simpler, our construction uses the same kinds of ingredients as Morita's definition:
it is based on the action of the group $\Hcob$  on the Malcev completion of $\pi$,
and it needs a ``symplectic expansion'' $\theta$ of the free group  $\pi$ which is also implicit in his paper \cite{morita_GD}. 
The result is a group $1$-cocycle $\Abel^\theta$ on $\Hcob$, whose definition  depends explicitly on $\theta$. 
Taking advantage of our simpler definition, 
we are able to prove some additional properties for this version of Morita's $1$-cocycle: 
the restriction $\IAbel$ of $\Abel^\theta$ to $\Hcyl $ is shown to be canonical (i.e$.$ independent of $\theta$), 
$\Hcob$-equivariant and invariant under stabilization of the surface $\Sigma$. 
We show that $\IAbel$ is dominated by  the tree-reduction of the LMO homomorphism discussed in \cite{massuyeau_IMH},
so that $\IAbel$ decomposes as an infinite sequence of finite-type invariants of strictly-increasing degrees.
Moreover, we relate $\IAbel$ to the determinant of the  Magnus representation of $\Hcyl$, 
which generalizes another result of Morita  for the Torelli group $\Torelli$ \cite{morita}. 
Finally, a result of Conant, Kassabov and Vogtmann \cite{CKV} reveals that
the target of our $1$-cocycle is much larger than Morita's original one. 

We apply our construction to the study of the abelianization of $\Hcob$.
It is known by a result of Cha, Friedl and Kim \cite{cfk} 
that the abelianization of $\Hcob$ for any $g \ge 1$ includes $(\Z/2\Z)^\infty$ as a direct summand. 
This was proved using the duality properties of the Reidemeister torsion for homology cobordisms, 
and left open the question of the triviality of the \emph{rational} abelianization of $\Hcob$.
By very different methods, we show  that $(\Hcob /[\Hcob,\Hcob])\otimes_{\Z} \Q \neq 0$ for all genus ${g \ge 1}$.
Indeed, using some algebraic results of Morita, Suzuki and the second-named author \cite{mss4,mss5},
we derive from the $1$-cocycle $\Abel^\theta$ a non-trivial group homomorphism $\widetilde{I}_{k}: {\Hcob \to \Q}$.
To be more specific, $ \widetilde{I}_{k}$ is a finite-type invariant of degree~$k$, 
with $k\in \{6,10\}$ for $g=1$ and $k=12$ for $g\geq 2$. 
Although it vanishes on the mapping class group $\mcg$, the  homomorphism 
$\widetilde{I}_{k}$ is determined by the action of $\Hcob$ 
on the Malcev completion of the $(k+1)$-st nilpotent quotient of $\pi$. 
(In particular, when $\Hcob= \Hcob_{\operatorname{smooth}}$, this homomorphism vanishes on the copy of $\Theta^3$ that $\Hcob$ naturally contains.) 

In fact, our approach is general enough to be applied to some other situations.
For instance, we will  consider the case of the automorphism group $\Aut(F)$ of a finitely-generated free group $F$.
The analogue of the Johnson filtration for this group is called the  {\it Andreadakis filtration}. 
As a side-product of our constructions, we obtain a comparison result between this filtration 
and the lower central series of the kernel of the canonical homomorphism $\Aut(F)\to \Aut(F/[F,F])$; 
the same result has been obtained around the same time by Bartholdi, with a rather similar strategy of  proof: 
see  \cite{bartholdi}. 

Besides, parallel to our study of the group of homology cobordisms~$\Hcob$,
we  may  have  also considered the concordance group  $\mathcal{S}:= \mathcal{S}_n$ 
of $n$-strand string-links in a thickened disk.
Indeed, each of the notions that we have mentioned 
so far for homology cobordisms has an exact analogue for string links:
\begin{center} 
\begin{tabular}{c|c}
group  $\Hcob$ & group  $\mathcal{S}$  \\
\hline
{\small mapping class group of $\Sigma$ } & {\small pure braid group on $n$ strands } \\
{\small Johnson homomorphisms } & {\small Milnor's $\mu$-invariants } \\
{\small Johnson filtration } & {\small Milnor filtration }\\
{\small  symplectic expansions }& {\small ``special'' expansions } {\footnotesize (see \cite{at,massuyeau_FDTO})} \\
{\small LMO homomorphism } &{\small Kontsevich integral }
\end{tabular} 
\end{center}
However, for the sake of brevity, we shall not  discuss string-links in this paper.\\

\noindent
{\bf Organization of the paper.} 
We begin by preparing some algebraic tools in Sections~\ref{sec:free}--\ref{sec:abel}. 
Our methods being infinitesimal by nature, we first recall some well-known facts  about Malcev Lie algebras of free groups.
Then we review the definition of Morita's trace and its generalization by Satoh,
proving that the latter defines a $1$-cocycle on the Lie algebra of derivations.
Next we introduce two maps on  the automorphism group of a complete free Lie algebra:
the ``abelianization map'' and a kind of  ``Magnus representation'', which are related one to the other by two sorts of trace maps. 
In Section \ref{sec:autFn}, we give a first application of our algebraic tools 
to the comparison of two filtrations of $\Aut (F)$. 
The topological applications start from Section \ref{sec:h_cob}: 
an infinitesimal version of the Dehn--Nielsen representation, combined to the above-mentioned  ``abelianization map'', 
gives rise to the $1$-cocycle  $\Abel^\theta$ 
which we relate to the LMO homomorphism.
The study of $\Abel^\theta$ continues in Section \ref{sec:trace_h_cob},
where the relation with the classical Magnus representation is established. 
In Section \ref{sec:abelian}, we use the $1$-cocycle  $\Abel^\theta$  to produce a non-trivial rational 
abelian quotient of the group $\Hcob$ for all $g \ge 1$. 
Appendix \ref{sec:cyclic_things} discusses a noncommutative version of the  log-determinant function.\\ 

\noindent
{\bf Notation and conventions.} Unless explicitly stated otherwise, the ground ring is the field $\Q$ of rational numbers.
Given a vector space $V$ (over $\Q$), we denote by $V^*$ the dual space $\Hom(V,\Q)$.
If $V$ is equipped with a decreasing filtration  of subspaces $(F_kV)_{k}$ indexed by the natural integers, then the vector space $V$ is given the corresponding topology
and its \emph{completion} is  $\widehat V :=  \varprojlim_{k}  V/F_kV$;
the filtration is said to be \emph{complete} if the canonical map $V \to \widehat V$ is an isomorphism.
If $V=\bigoplus_{i\geq 0} V_i$ is a graded vector space,
then it is equipped with the degree-filtration $(V_{\geq k})_k$ where $V_{\geq k}:=\bigoplus_{i\geq k} V_i$: 
the corresponding completion $\widehat V = \prod_{i \geq 0} V_i $ 
is called  the \emph{degree-completion}.\\

\noindent
{\bf Acknowledgements.} 
Some of the results of this paper have been obtained in 2011 
while the second-named author visited the University of Strasbourg, whose hospitality is gratefully acknowledged. 
The authors would like to thank Athanase Papadopoulos and Ken-ichi Ohshika for giving 
chances to keep on working through the workshops held in 2014 and 2015 
in the framework of a CNRS--JSPS joint seminar. 
Thanks are also due to  Nariya Kawazumi for informing the authors 
that he also  studied some commutative analogues of  
the ``Magnus representations'' that are introduced in  Section~\ref{subsec:Magnus}. 
The second-named author would like to thank Shigeyuki Morita and Masaaki Suzuki for helpful discussions 
on the topics in Section~\ref{sec:abelian}. 
He was partially supported by KAKENHI 
(No.~15H03618 and No.~15H03619),  Japan Society for the Promotion of Science, Japan.

\section{Malcev Lie algebras of free groups}  \label{sec:free}

In this section, we briefly review some well-known material related to free groups and their Malcev Lie algebras.
The reader may consult \cite[Section~2]{massuyeau_IMH} for further informations and references. 
Next, we specialize this material to  the ``symplectic case'', which we shall need to study  groups of homology cobordisms.

\subsection{Malcev Lie algebras} \label{subsec:malcev}

Let $G$ be a group. The group algebra $\Q[G]$ of $G$ has a canonical Hopf algebra structure 
with coproduct $\Delta$, counit $\varepsilon$ and antipode $S$ defined by $\Delta(g)= g \otimes g$, $\varepsilon(g)=1$ and $S(g)=g^{-1}$, respectively, for any $g\in G$.
Let $I:= \ker(\varepsilon)$ denote the augmentation ideal of $\Q[G]$.
The $I$-adic completion 
$$
\widehat{\Q[G]} := \varprojlim_{k} \Q[G]/I^k
$$
of $\Q[G]$ is a complete Hopf algebra in the sense of \cite[Appendix A]{Quillen_rational},
whose coproduct $\hat \Delta$, counit $\hat \varepsilon$ and antipode  $\hat S$ are induced by $\Delta, \varepsilon$ and $S$ respectively.

By definition,  the \emph{Malcev Lie algebra} of $G$ is  the Lie algebra of primitive elements of $\widehat{\Q[G]}$, 
i.e$.$ we have  
$$
\Malcev(G) := \big\{ x\in \widehat{\Q[G]} : \hat \Delta(x) = x \hat \otimes 1 + 1 \hat \otimes x\big\}.
$$
The completed  $I$-adic filtration on $\widehat{\Q[G]}$ restricts to a filtration on the Lie algebra $\Malcev(G)$.
We recall how it is related  to the \emph{lower central series} of $G$
$$
G= \Gamma_1 G \supset \Gamma_2 G \supset \cdots \supset \Gamma_{k-1} G \supset \Gamma_{k} G \supset \cdots
$$
which is defined inductively by $\Gamma_k G := [\Gamma_{k-1} G, G]$ for any integer $k\geq 2$.
The logarithmic series $\log: G \to \Malcev(G)$ defined on any $g\in G$ by $\log(g) := \sum_{k\geq 1} (-1)^{k+1} (g-1)^k/k$ preserves the filtration
and it mainly follows from \cite{quillen} that $\log$ induces an isomorphism at the graded level:
\begin{equation} \label{eq:Gr_log}
(\Gr \log) {\otimes_\Z\,} \Q: (\Gr G) {\otimes_\Z\,} \Q \stackrel{\simeq}{\longrightarrow} \Gr \Malcev(G) 
\end{equation}
Note that both $(\Gr G) {\otimes_\Z\,} \Q$  and $\Gr \Malcev(G)$   are  graded Lie algebras 
(whose Lie brackets are respectively induced by the commutator in $G$ and the Lie bracket in $ \Malcev(G)$),
and  the map $(\Gr \log) {\otimes_\Z\,} \Q$ preserves these structures.

\subsection{Expansions of free groups} \label{subsec:expansion}

We now consider the case of a free group $F$ of finite rank $n$.
We set $H:=(F/[F,F]) \otimes_\Z \Q$ and denote by $\freeLie := \freeLie(H) $ the graded  Lie algebra  freely generated by $H$ in degree $1$:
for any $i\geq 1$, its degree $i$ part $\freeLie_i := \freeLie_i(H)$ consists of iterated Lie brackets in $H$ of length $i$. 
We recall how the Malcev Lie algebra $\Malcev(F)$ can be identified by means of ``expansions'' 
to the degree-completion $\widehat\freeLie := \widehat\freeLie(H)$ of $\freeLie$.

Let  $T:=T (H) = \bigoplus_{i \ge 0} H^{\otimes i}$ be the graded associative algebra freely generated by $H$ in degree $1$, 
and let $\widehat T := \widehat T (H)$ denote the degree-completion of $T$. 
Note that $\widehat T$ is a complete Hopf algebra whose Lie algebra of primitive elements is~$\widehat \freeLie$. 
An {\it expansion} of  $F$ is a  map 
$\theta: F \to \widehat{T}$ which satisfies $\theta(xy) =\theta(x) \theta(y)$ for any $x,y\in F$,
and such that
\begin{equation} \label{eq:1+x+etc}
\theta (x)=1+[x]+\text{(higher-degree terms)}
\end{equation}
for any $x \in F$ with homology class $[x] \in H$. 
The expansion $\theta$ is   {\it group-like} if it takes group-like values. 

\begin{example}\label{ex:expansions}
Let $\gamma:=(\gamma_1,\dots, \gamma_n)$ be a basis of $F$. 
Classically, the {\it Magnus expansion}  refers to the expansion $\theta_\gamma$ of $F$ defined by $\theta_\gamma (\gamma_i) :=1+[\gamma_i]$.
 Another example is the expansion $\theta_\gamma^{\exp}$ of $F$  given by 
\[\theta_\gamma^{\exp} (\gamma_i): = \exp([\gamma_i])= \sum_{k=0}^\infty  \frac{[\gamma_i]^k}{k!} .\]
Note that, contrary to $\theta_\gamma$, the expansion $\theta_\gamma^{\exp}$ is group-like. \hfill $\blacksquare$
\end{example}

Any expansion $\theta: F \to \widehat T$ extends to a complete algebra homomorphism $\widehat \theta: \widehat{\Q[F]} \to   \widehat T$
and, by condition \eqref{eq:1+x+etc}, $\widehat \theta$ is an isomorphism. If $\theta$ is group-like, then $\widehat \theta$ preserves the Hopf algebra structures, 
so that it restricts to a filtration-preserving Lie algebra isomorphism
\begin{equation}\label{eq:theta_Lie}
\widehat{\theta}: \Malcev(F)  \stackrel{\simeq}{\longrightarrow} \widehat \freeLie.
\end{equation}
Note that the composition of $\Gr \widehat{\theta}: \Gr \Malcev(F) \to \Gr  \widehat\freeLie = \Gr \freeLie= \freeLie$ 
with the isomorphism $(\Gr \log ) {\otimes_\Z\,} \Q: (\Gr F) {\otimes_\Z\,} \Q \to \Gr \Malcev(F)$ is the rational form of the  canonical  isomorphism of graded Lie rings
\begin{equation}   \label{eq:free_free}
\Gr F =\bigoplus_{k=1}^\infty \frac{\Gamma_k F}{ \Gamma_{k+1} F}
\stackrel{  \simeq } {\longrightarrow}\freeLie^\Z  = \bigoplus_{k=1}^\infty \freeLie_k^\Z
\end{equation}
given by the identity in degree $1$.
Here $\freeLie^\Z :=\freeLie^\Z (H^\Z)$ denotes the graded Lie ring freely generated by $H^\Z:= F/[F,F]$ in degree $1$.

\subsection{Symplectic expansions} \label{subsec:symp_expansion}

We now assume that $F$ is  the fundamental group $\pi:=\pi_1 (\Sigma,\star)$ of a compact oriented connected surface $\Sigma$  
of genus $g$ with one boundary component. The base point $\star$ is fixed on $\partial  \Sigma$. 
The homology intersection on $\Sigma$ defines a symplectic form on $H_1(\Sigma;\Q)\simeq H$ which, by the resulting duality $H \simeq H^*$, 
can also be regarded as an element $\omega \in \wedge^2 H \subset H^{\otimes 2}$.
An expansion $\theta$ of $\pi$ is  {\it symplectic} if it is group-like and  maps 
$\zeta := [\partial \Sigma] \in \pi$ to $\exp({-\omega})$.
(Here the boundary curve $\partial \Sigma$ has the orientation inherited from $\Sigma$.)

Let $\gamma :=  (\alpha_1,\dots,\alpha_g,\beta_1,\dots,\beta_g)$ be  a basis of $\pi$ given by a system of meridians and parallels on the surface $\Sigma$.
Then the expansion $\theta^{\exp}_\gamma$ of Example~\ref{ex:expansions} can be modified degree-after-degree to produce an example of symplectic expansion \cite[Lemma 2.16]{massuyeau_IMH}.
(See also \cite[Lemma~3.6]{morita_GD} for a similar result and \cite{kuno} for a more general construction.)

\section{Morita's trace map and its variants}  \label{sec:traces}

In this section, we review Morita's trace map  on the Lie algebra of derivations of a free Lie algebra \cite{morita} following the generalization proposed by Satoh \cite{satoh}.
 We also present some  variants of this construction.

\subsection{Automorphisms of free Lie algebras} \label{subsec:automorphisms}

Let $\freeLie:=\freeLie(H)$ be the graded free Lie algebra generated by a finite-dimensional $\Q$-vector space  $H$ in degree~$1$.
We denote by $\Aut(\widehat{\freeLie})$ the group of filtration-preserving automorphisms of the  degree-completion $\widehat{\freeLie}$.

Let $\mathfrak{R}$ be a characteristic ideal of the Lie algebra ${\freeLie}$ and denote by $\widehat{\mathfrak{R}}$ the closure of $\mathfrak{R}$ in $\widehat{\freeLie}$.
We denote by $\Aut^{{\mathfrak{R}}}(\widehat{\freeLie})$ the subgroup of $\Aut(\widehat{\freeLie})$
consisting of the automorphisms that induce the identity at the level of $\widehat{\freeLie}/\widehat{\mathfrak{R}}$.
In particular, for $\mathfrak{R}:= \freeLie_{\geq 2}$, we have $\widehat{\mathfrak{R}}= \widehat{\freeLie}_{\ge 2} $ so that  $\Aut^{{\mathfrak{R}}}(\widehat{\freeLie})$
is the group $\IAut(\widehat{\freeLie})$ of  automorphisms that induce the identity at the graded level. 
There is a short exact sequence of groups
\begin{equation}   \label{eq:IAut(L)}
\xymatrix{
1 \ar[r] & \IAut(\widehat{\freeLie}) \ar[r]  & \Aut(\widehat{\freeLie})  \ar[r]^\sigma & \GL(H) \ar[r] & 1,
}
\end{equation}
where the general linear group $\GL(H)$ is identified with the group of 
automorphisms of $ \widehat{\freeLie}/\widehat{\freeLie}_{\geq 2}$.

\subsection{Derivations of free Lie algebras} \label{subsec:derivations}

We denote by  $\Der(\widehat{\freeLie}, \widehat{\freeLie})$ 
the Lie algebra of filtration-preserving derivations of $\widehat{\freeLie}$.
It contains, as an ideal, the Lie algebra $\Der(\widehat{\freeLie}, \widehat{\freeLie}_{\ge 2})$ of derivations with values 
in $\widehat{\freeLie}_{\ge 2}$. It is well-known that the logarithmic series 
$$
\log(\psi) = \sum_{n =1}^\infty (-1)^{n+1}  \frac{ \big(  \psi-  \Id_{\widehat{\freeLie}} \big)^n }{n} 
$$
establishes a one-to-one correspondence $\IAut(\widehat{\freeLie}) \stackrel{\simeq }{\longrightarrow}\Der(\widehat{\freeLie}, \widehat{\freeLie}_{\ge 2})$
whose inverse map is given by the exponential series
$$
\exp(\delta) =  \sum_{n=0}^\infty  \frac{\delta^n}{n!}. 
$$
(See \cite[Proposition 5.12]{massuyeau_IMH}, for instance.)
Furthermore, for any characteristic ideal $\mathfrak{R}$ of  ${\freeLie}$ contained in ${\freeLie}_{\ge 2}$, 
this correspondence sends the subgroup $\Aut^{{\mathfrak{R}}}(\widehat{\freeLie})$ 
to the Lie subalgebra $\Der(\widehat{\freeLie}, \widehat{\mathfrak{R}})$ of $\widehat{\mathfrak{R}}$-valued derivations.

The canonical action of $\GL(H)$ on $H$ extends to an action of $\GL(H)$ on the Lie algebra $\widehat{\freeLie}$.
Hence we can regard $\GL(H)$ as a subgroup of $\Aut(\widehat \freeLie)$,
which provides a canonical section to the short exact sequence \eqref{eq:IAut(L)}.
Besides, there is a canonical action of   $\GL(H)$ 
on the Lie algebra $\Der(\widehat{\freeLie}, \widehat{\freeLie})$ given by $(A,\delta) \mapsto A \circ \delta \circ A^{-1}$.
Clearly, for any  characteristic ideal $\mathfrak{R} \subset{\freeLie}_{\ge 2}$,
the Lie subalgebra  $\Der(\widehat{\freeLie}, \widehat{\mathfrak{R}})$ is preserved by this action.

Clearly a filtration-preserving derivation of $\widehat \freeLie$ is determined by its restriction to $H= \freeLie_1$.
Hence we can identify the vector space $\Der(\widehat{\freeLie}, \widehat{\freeLie})$ to $\Hom(H,\widehat\freeLie)$
and, similarly,  the vector space $\Der({\freeLie}, {\freeLie})$ of  derivations of $\freeLie$ can be identified to  $\Hom(H,\freeLie)$.
Therefore, by the decomposition
$$
\Hom(H,\widehat\freeLie) =  \prod_{d=0}^\infty  \Hom(H, \freeLie_{d+1}), 
$$
we can regard $\Der(\widehat{\freeLie}, \widehat{\freeLie})$ as the degree-completion of $\Der(\freeLie, \freeLie)$. 
Here a derivation $\delta: \freeLie \to \freeLie$ is  homogeneous of \emph{degree} $d$ if it sends $H$ to $\freeLie_{d+1}$.

\subsection{The trace cocycle} \label{subsec:trace}

Let $T:= T(H)$ and denote by $[T,T]$ the subspace of $T$ spanned by commutators $[u,v]=uv-vu$, for all $u,v \in T$.
We consider the quotient space $C(H):= T/[T,T]$. 
Thus, for any integer $k\geq 1$, the degree $k$ part $C_k(H)$ of $C(H)$ consists of homogeneous tensors in $H$ of degree $k$ up to cyclic permutation of the components. 

Following Satoh \cite{satoh}, we consider the degree-preserving $\GL(H)$-equivariant linear map 
$$
 \Trace: \Der(\freeLie,\freeLie) \longrightarrow {C}(H) 
$$
which, in degree $k\geq 0$, is defined by the following composition: 
$$
\xymatrix{
\Hom (H, \freeLie_{k+1}) \simeq  H^\ast \otimes \freeLie_{k+1} \subset 
H^\ast \otimes H^{\otimes (k+1)} \ar[r]^-{\operatorname{ev}} & H^{\otimes k} \ar[r]^-{\operatorname{proj}} & C_k(H).
}
$$
Here, the map ``ev'' is the tensor product of the evaluation map $H^* \otimes H \to \Q$ with the identity of $H^{\otimes k}$,
and the map ``proj'' is the canonical projection. 

\begin{example} \label{ex:trace_as_trace}
Let $\delta \in \Hom(H,H)$ and regard it as a derivation of $\freeLie$ of degree $0$. 
Then $\Trace(\delta) \in C_0(H) \simeq \Q$ is the usual trace $\trace(\delta)$ of the endomorphism~$\delta$. \hfill $\blacksquare$
\end{example}

In general, the ``trace'' $\Trace(\delta)$ of a $\delta \in \Der(\freeLie,  \freeLie)$ can be  explicitly computed as follows.
Let $n:= \dim(H)$ and fix a basis $(x_1,\dots,x_n)$ of $H$. 
We define the linear maps $\partial_i: T \to T$ for all $i\in \{1,\dots, n\}$ by requiring that
\begin{equation} \label{eq:fundamental}
\forall v \in T, \quad v - \varepsilon(v) =  \sum_{k=1}^n x_k\,  \partial_k(v).
\end{equation}
where $\varepsilon:T \to \Q$ is the counit defined by 
$\varepsilon(q \cdot 1)=q$ for all $q\in \Q$ and $\varepsilon(u)=0$ for all $u\in T_{\geq 1}$.
It follows from \eqref{eq:fundamental} that  $\partial_i$ is a ``Fox derivation'' of the augmented algebra $(T,\varepsilon)$, i.e$.$
\begin{equation} \label{eq:Fox}
\forall u, v \in T, \quad \partial_i(uv) = \partial_i(u)\, v + \varepsilon(u)\, \partial_i(v).
\end{equation}
If $(x_1^*,\dots,x_n^*)$ denotes the basis of $H^*$ dual to $(x_1,\dots,x_n)$, then 
$\partial_i(u):= \operatorname{ev}(  { x_i^*\otimes u })\in T$ for any $u\in T$. 
Thus we obtain the ``divergence'' formula
\begin{equation} \label{eq:divergence}
\Trace(\delta) = \sum_{i=1}^n \operatorname{ev}\big( x_i^* \otimes \delta(x_i) \big) =  \sum_{i=1}^n \partial_i\big( \delta(x_i)\big)
\end{equation}
for any $\delta \in \Der(\freeLie,  \freeLie) \simeq \Hom(H,\freeLie)$.

Observe that any $\delta \in \Der(\freeLie,  \freeLie)$  extends uniquely to a  derivation $\widetilde{\delta}$ of  the algebra~$T$,
and clearly $\widetilde{\delta}$ preserves the subspace  $[T,T]$. This defines an action of the Lie algebra $\Der(\freeLie,\freeLie)$ on the space ${C}(H)$.
The proof of the next lemma is delayed to  the end of this section. 

\begin{lemma} \label{lem:trace_cocycle}
The map $\Trace$ is a $1$-cocycle of the Lie algebra $\Der( \freeLie,  \freeLie)$: 
for any $\delta, \eta \in  \Der( \freeLie,  \freeLie)$,  we have 
\begin{equation} \label{eq:1-cocycle}
\Trace\left(\left[\delta, \eta\right]\right) = \delta \cdot \Trace(\eta) -  \eta \cdot \Trace(\delta).
\end{equation}
\end{lemma}

It has been proved by Satoh \cite{satoh} that  $\Trace:  \Der(\freeLie,\freeLie) \to {C}(H)$ 
is surjective in any fixed degree $d$ if $\dim(H)$ is large enough with respect to $d$ 
(see Theorem \ref{thm:satoh} below). 
Let $\mathfrak{R} \subset {\freeLie}_{\geq 2}$ be a characteristic ideal and let $C^{\mathfrak{R}}(H)$ be the quotient of $T$ by the subspace $[T,T]+ T\mathfrak{R}T$, 
where  $T\mathfrak{R}T$ is the two-sided ideal of $T$ generated by~$\mathfrak{R}$. 
The composition of $\Trace$ with the canonical projection ${C}(H) \to C^{\mathfrak{R}}(H)$ is denoted by $\Trace^{\mathfrak{R}}$.
It follows from Lemma~\ref{lem:trace_cocycle} that $\Trace^{\mathfrak{R}}$ vanishes on the image of the Lie bracket of $\Der(\freeLie,\mathfrak{R})$,
so that it induces a degree-preserving $\GL(H)$-equivariant  linear map
$$
\Trace^{\mathfrak{R}}: H_1\big(\Der(\freeLie,\mathfrak{R})\big) \longrightarrow C^{\mathfrak{R}}(H).
$$

An important special case  is provided by $\mathfrak{R}:= \freeLie_{\geq 2}$.
Then we have $C^{\mathfrak{R}}(H)=S(H)$ and the resulting map $\Trace^{\mathfrak{R}}$ is denoted by 
\begin{equation} \label{eq:Morita_trace}
\ITrace: H_1\big(\Der( \freeLie,  \freeLie_{\geq 2})\big) \longrightarrow {S}(H).
\end{equation}
The map $\ITrace$,  which we describe in more detail below, was originally defined by Morita \cite{morita} in his study of mapping class groups.

\subsection{The symplectic case}

Assume now that the vector space $H$ is equipped with a symplectic form $\omega: H \times H \to \Q$.
We shall  identify $H$ to $H^*$ via the isomorphism $h\mapsto \omega(h,-)$.
In this case, the contents of the previous subsections can be specified as follows.

We  identify $\Der(\widehat{\freeLie}, \widehat{\freeLie})$
to $\Hom(H, \widehat\freeLie) = H^* \otimes \widehat\freeLie \simeq H \otimes \widehat\freeLie$ using the above isomorphism $H\simeq H^*$.
Furthermore, denoting by  $\omega \in \wedge^2 H = \freeLie_2$  the bivector dual to $\omega \in \wedge^2 H^*$,
we consider the subgroups
$$
\Aut_\omega(\widehat \freeLie) \subset \Aut(\widehat \freeLie), \quad \IAut_\omega(\widehat \freeLie) \subset \IAut(\widehat \freeLie)
$$
of automorphisms that fix $\omega$, and the Lie subalgebras
$$
\widehat{\mathfrak{h}}:=\Der_\omega (\widehat{\freeLie}, \widehat{\freeLie}) \subset \Der(\widehat{\freeLie}, \widehat{\freeLie}), \quad 
\widehat{\mathfrak{h}}^+:=\Der_\omega (\widehat{\freeLie}, \widehat{\freeLie}_{\geq 2}) \subset \Der(\widehat{\freeLie}, \widehat{\freeLie}_{\geq 2})
$$
of derivations that vanish on $\omega$.
Then the short exact sequence \eqref{eq:IAut(L)} specializes to
\begin{equation}   \label{eq:IAut(L)_symplectic}
\xymatrix{
1 \ar[r] & \IAut_\omega(\widehat{\freeLie}) \ar[r]  & \Aut_\omega(\widehat{\freeLie})  \ar[r]^-\sigma & \Sp(H) \ar[r] & 1,
}
\end{equation}
where $\Sp(H)$ denotes the group of symplectic automorphisms of $H$.
Furthermore, the  correspondence  between $\IAut(\widehat{\freeLie})$ and $\Der(\widehat{\freeLie}, \widehat{\freeLie}_{\ge 2})$ described in Section \ref{subsec:derivations}
makes $\IAut_\omega (\widehat{\freeLie})$ correspond to~$\widehat{\mathfrak{h}}^+$.

By restriction to the Lie subalgebra ${\mathfrak{h}}^+ \subset  \Der({\freeLie}, {\freeLie}_{\geq 2})$, 
Morita's trace map \eqref{eq:Morita_trace}  induces an $\Sp(H)$-equivariant  degree-preserving linear map
\begin{equation} \label{eq:authentic_Morita}
\ITrace : {H_1}\big({\mathfrak{h}}^+\big) \longrightarrow {S}(H).
\end{equation}
Here the action of $\Sp(H)$  on the Lie algebra ${\mathfrak{h}}^+$  is the restriction of 
 the $\GL(H)$-action on $\Der(\freeLie,\freeLie_{\geq 2})$  described in Section \ref{subsec:derivations}. 
According to  \cite[Theorem~6.1]{morita}, we have  the following when $\dim(H) > 2$: 
\begin{eqnarray}
\label{eq:surjectivity} && \hbox{$\ITrace_{2k+1}:  {H_1}\big({\mathfrak{h}}^+\big)_{2k+1} \longrightarrow {S}^{2k+1}(H)$  is surjective for any $k\geq 0$;} \\
\label{eq:nullity} && \hbox{$\ITrace_{2k}:  {H_1}\big({\mathfrak{h}}^+\big)_{2k} \longrightarrow {S}^{2k}(H)$  is trivial for any $k\geq 1$.}
\end{eqnarray}
(Note that $\ITrace_k=0$ for all $k \ge 1$ when $\dim(H)=2$,
 see  \cite[Proposition~8.2]{mss4}, for instance.)
It~has also been shown by Nakamura in an unpublished work 
that, for any  $k\geq 0$, the $\Sp(H)$-module $\mathfrak{h}_{2k+1}$ has only one copy 
$S^{2k+1}(H)$ in its irreducible decomposition \cite{es}. Therefore the projection 
$H_1 (\mathfrak{h}^+)_{2k+1} \to S^{2k+1}(H)$ provided by $\ITrace$ in degree $2k+1$ is unique up to multiplication by a non-zero scalar. 

Passing the map \eqref{eq:authentic_Morita} to the degree-completions, we 
obtain an $\Sp(H)$-equivariant filtration-preserving linear map
$$
\ITrace : \widehat{H_1}\big(\widehat {\mathfrak{h}}^+\big) \longrightarrow \widehat{S}(H).
$$
Here  $\widehat{H_1}\big(\widehat{\mathfrak{h}}^+\big)$ denotes the degree-completion of ${H_1}\big({\mathfrak{h}}^+\big)$
or, equivalently, it is  the quotient of 
$\widehat{\mathfrak{h}}^+$  by  the closure of 
the image $\big[\widehat{\mathfrak{h}}^+,\widehat{\mathfrak{h}}^+\big]$ of the Lie bracket.

\subsection{Proof of Lemma \ref{lem:trace_cocycle}}

The lemma follows from \cite[Propositions 3.19 \& 3.20]{at} if one restricts the map $\Trace$  to a certain subalgebra of $\Der(\freeLie,  \freeLie)$,
namely the Lie algebra of ``tangential derivations'' in the terminology of \cite{at}. The proof  of the lemma in the general case uses the same kind of arguments and proceeds as follows.
In the sequel,  we will omit the sums over repeated indices, and we denote by a dot $\cdot$ the action of $\Der(\freeLie,  \freeLie)$ on $\freeLie$, $T$ or $C(H)$.

In order to prove \eqref{eq:1-cocycle} for any $\delta, \eta \in  \Der( \freeLie,  \freeLie)$,
we can assume that  $\delta$ and $\eta$ are homogeneous of degree $d$ and $e$ respectively.
We consider firstly the case where $e=0$. Let $(n_{ij})_{i,j}$ be the matrix of $\eta\in \Hom(H,H)$ in the basis $(x_i)_i$: i.e. $\eta(x_i) = n_{ij} x_j$ for all $i$. Then
\begin{eqnarray*}
\Trace([\delta,\eta]) &=& \partial_i \left([\delta,\eta](x_i)\right) \\
&=&   \partial_i    \left(\delta(\eta(x_i))\right) -   \partial_i  \left(\eta\cdot \delta(x_i)\right)  \\
&=&    n_{ij} \operatorname{ev} \left(x_i^*\otimes  \delta(x_j)\right) -  \operatorname{ev} \left(x_i^*\otimes  \eta\cdot \delta(x_i)\right) \\
&=&  \operatorname{ev} \Big( \big( n_{ij} x_i^*\big) \otimes  \delta(x_j)\Big) -   \operatorname{ev} \left(x_i^*\otimes  \eta\cdot \delta(x_i)\right) \\
&=&    \operatorname{ev} \big( \eta^*(x_j^*) \otimes  \delta(x_j)\big) -  \operatorname{ev} \left(x_i^*\otimes  \eta\cdot \delta(x_i)\right).
\end{eqnarray*}
Here $\eta^* \in \Hom(H^*,H^*)$ is the dual of the endomorphism $\eta$. 
Hence 
$$
\Trace([\delta,\eta]) = - \eta\cdot  \operatorname{ev}(x_i^* \otimes \delta(x_i)) = - \eta \cdot \Trace(\delta) = \delta \cdot \Trace(\eta) -  \eta \cdot \Trace(\delta), 
$$
where the last identity follows from the fact that the action of $\Der(\freeLie,\freeLie)$ on the degree $0$ part of $C(H)$ is trivial.  

The case $d=0$ follows from the case $e=0$ by skew-symmetry of the Lie bracket in $\Der(\freeLie, \freeLie)$. 
Hence we  now assume that $d\geq 1$ and $e\geq 1$. For every  $i\in \{1,\dots,n\}$, we write
$$
\delta(x_i) = \sum_{j=1}^n [x_j, u_{ji}], \quad \eta(x_i) = \sum_{j=1}^n  [x_j,v_{ji}] 
$$ 
where $u_{ji} \in \freeLie_d$ and $v_{ji} \in \freeLie_e$. Hence we obtain
\begin{eqnarray*}
\Trace(\eta) \ = \   \partial_i\big( \eta(x_i)\big) \ = \   \partial_i\left( [x_j, v_{ji}]\right) 
 &=&    \partial_i (x_j v_{ji}) -    \partial_i (v_{ji} x_j)\\
& \by{eq:Fox} &     v_{ii} -   \partial_i (v_{ji}) x_j
\end{eqnarray*}
and we deduce that
\begin{equation} \label{eq:delta_eta}
\delta \cdot \Trace(\eta) =    \delta\cdot  v_{ii} -   \left( \delta\cdot \partial_i (v_{ji})\right)\,  x_j -    \partial_i (v_{ji})\,  \delta(x_j).
\end{equation}
Besides, we have
\begin{equation} \label{eq:[delta,eta]}
\Trace([\delta,\eta])  =  \partial_i([\delta,\eta](x_i)) =   X_{\delta,\eta} - X_{\eta,\delta}
\end{equation}
where $X_{\delta,\eta} :=   \partial_i \left(\delta \cdot \eta(x_i)\right)$ and $X_{\eta,\delta}$ is similarly defined. We compute the former:
\begin{eqnarray*}
X_{\delta,\eta} &=&   \partial_i    \left(\delta \cdot [x_j,v_{ji}]\right) \\
 &=&  \partial_i    \left(\delta \cdot (x_j  v_{ji})\right) -  \partial_i    \left(\delta \cdot (v_{ji} x_{j})\right) \\
&=&  \partial_i    \left(\delta(x_j)\,  v_{ji}\right) +   \partial_i    \left(x_j\,  (\delta \cdot  v_{ji})\right) -  \partial_i    \left((\delta \cdot v_{ji})\,  x_{j}\right) -  \partial_i    \left(v_{ji}\, \delta(x_{j})\right) \\
&\by{eq:Fox}&  \partial_i    \left(\delta(x_j)\right)  v_{ji} +  \delta \cdot  v_{ii} -  \partial_i    \left(\delta \cdot v_{ji}\right)\,  x_{j} -  \partial_i    \left(v_{ji}\right)\, \delta(x_{j}).
\end{eqnarray*}
The third summand in the last equation is 
\begin{eqnarray*}
\partial_i    \left(\delta \cdot v_{ji}\right)\,  x_{j}
 &\by{eq:fundamental}&  \partial_i    \left(\delta \cdot \big(x_k \partial_k(v_{ji}) \big) \right)\,  x_{j} \\
 &=&  \partial_i    \big(\delta (x_k)\, \partial_k(v_{ji}) \big)\,  x_{j}  +  \partial_i    \big(x_k\, (\delta\cdot \partial_k(v_{ji})) \big)\,  x_{j} \\
&=&  \partial_i  \big(  \delta (x_k)\big)\, \partial_k(v_{ji}) \,  x_{j}  +   (\delta\cdot \partial_i(v_{ji}))\,  x_{j} .
\end{eqnarray*}
Hence, using \eqref{eq:delta_eta}, we deduce that
\begin{eqnarray*}
&& X_{\delta,\eta} -  \delta \cdot \Trace(\eta)  \\
&=&   \partial_i    \left(\delta(x_j)\right)  v_{ji}   -  \partial_i \big(\delta (x_k)\big)\, \partial_k(v_{ji}) \,  x_{j} \\
&=&   \partial_i    \left([x_r,u_{rj}]\right)  v_{ji}   -  \partial_i([x_r,u_{rk}])\, \partial_k(v_{ji}) \,  x_{j} \\
&=&  \partial_i    \left(x_ru_{rj}\right)  v_{ji}   -   \partial_i    \left(u_{rj}x_r\right)  v_{ji}   \\
&&  -  \partial_i(x_r u_{rk})\, \partial_k(v_{ji}) \,  x_{j}    + \partial_i(u_{rk} x_r)\, \partial_k(v_{ji}) \,  x_{j}  \\
&=&   u_{ij} v_{ji}   -  \partial_i    \left(u_{rj}\right) x_r  v_{ji}   -   u_{ik} \partial_k(v_{ji}) \,  x_{j}    +   \partial_i(u_{rk}) x_r \partial_k(v_{ji}) \,  x_{j} .
\end{eqnarray*}
We conclude that
\begin{eqnarray*}
&&\Trace([\delta,\eta]) -  \delta \cdot \Trace(\eta)  +  \eta \cdot \Trace(\delta)  \\
&\by{eq:[delta,eta]}& \left(X_{\delta,\eta} -  \delta \cdot \Trace(\eta) \right) - \left(X_{\eta,\delta} -  \eta \cdot \Trace(\delta)  \right) \\
&=&   u_{ij} v_{ji}   -  \partial_i    \left(u_{rj}\right) x_r  v_{ji}     -   u_{ik} \partial_k(v_{ji}) \,  x_{j}    +  \partial_i(u_{rk}) x_r \partial_k(v_{ji}) \,  x_{j} \\
&& - v_{ij} u_{ji}   +  \partial_i    \left(v_{rj}\right) x_r  u_{ji}     +   v_{ik} \partial_k(u_{ji}) \,  x_{j}    -  \partial_i(v_{rk}) x_r \partial_k(u_{ji}) \,  x_{j}  \\
&=&   u_{ij} v_{ji}   -  \partial_i    \left(u_{rj}\right) x_r  v_{ji}     -   u_{ik} \partial_k(v_{ji}) \,  x_{j}    +  \partial_i(u_{rk}) x_r \partial_k(v_{ji}) \,  x_{j} \\
&& - u_{ji}  v_{ij}   + u_{ji}  \partial_i    \left(v_{rj}\right) x_r       +    \partial_k(u_{ji}) \,  x_{j}  v_{ik}   -   \partial_k(u_{ji}) \,  x_{j} \partial_i(v_{rk}) x_r \\
&=& 0
\end{eqnarray*}
where, in the penultimate identity, we have used the fact that $\Trace$ takes values in the quotient $C(H)$  of $T(H)$.

\section{The abelianization map and the Magnus representation}\label{sec:abel}

Let $H$ be a vector space of finite dimension $n$ and let $\widehat\freeLie:= \widehat\freeLie(H)$.
We define a kind of ``abelianization'' map on $\Aut(\widehat\freeLie)$ and we relate this to 
the ``Magnus representation'' of $\Aut(\widehat\freeLie)$.

\subsection{The abelianization map} \label{subsec:abelianization}

Recall that  $\sigma : \Aut(\widehat\freeLie) \to \ \Aut(\widehat\freeLie/\widehat\freeLie_{\geq 2}) \simeq \GL(H)$ 
is the canonical homomorphism, and that the canonical action of  $\GL(H)$  on $\widehat{\freeLie}$ defines an embedding $\GL(H) \hookrightarrow \Aut(\widehat\freeLie)$.
Consider the map 
$$
\Abel: \Aut(\widehat\freeLie)  \longrightarrow \widehat{H_1}\big(\Der(\widehat\freeLie,\widehat \freeLie_{\geq 2})\big), \ \psi \longmapsto \log\left(\psi\, \sigma(\psi)^{-1}\right)
$$
where  $\widehat{H_1}\big(\Der(\widehat\freeLie,\widehat\freeLie_{\geq 2})\big)$
is  the quotient of the complete Lie algebra $\Der(\widehat\freeLie,\widehat\freeLie_{\geq 2})$ 
by the closure of the image of its Lie bracket; 
the restriction of the  map $\Abel$ to $\IAut(\widehat\freeLie)$ is denoted by $\IAbel$. 
The group $\Aut(\widehat\freeLie)$ acts on the space $\widehat{H_1}\big(\Der(\widehat\freeLie,\widehat\freeLie_{\geq 2})\big)$ via $\sigma$,
and it also acts on $\IAut(\widehat\freeLie)$ by conjugacy. 

\begin{lemma} \label{lem:abelianization}
$(1)$ The map $\Abel$ is a group $1$-cocycle: for all $\psi, \varphi \in  \Aut(\widehat\freeLie)$, we have 
$$
\Abel(\psi \varphi) = \Abel(\psi) + \sigma(\psi)\cdot \Abel(\varphi).
$$
$(2)$ The map $\IAbel$  is an  $\Aut(\widehat\freeLie)$-equivariant group homomorphism.
\end{lemma}

\begin{proof}
(1) Let $\psi, \varphi\in \Aut(\widehat\freeLie)$. We have
\begin{align*}
\log \left( \psi\varphi\, \sigma(\psi \varphi)^{-1}\right) 
&= \log \left(\psi \varphi\, \sigma(\varphi)^{-1} \sigma(\psi)^{-1}\right)\\
&= \log \left(\psi  \sigma(\psi)^{-1} \sigma(\psi)  \varphi \sigma(\varphi)^{-1} \sigma(\psi)^{-1}\right)\\ 
&\equiv \log\left(\psi\sigma(\psi)^{-1}  \right) +  \log\left( \sigma(\psi) \varphi  \sigma(\varphi)^{-1} \sigma(\psi)^{-1} \right) \\
&\equiv \Abel(\psi) + \sigma(\psi) \circ  \log\left(  \varphi  \sigma(\varphi)^{-1} \right) \circ \sigma(\psi)^{-1} \\
&\equiv  \Abel(\psi)  + \sigma(\psi) \cdot \Abel(\varphi)
\end{align*}
where  $\equiv$ denotes a congruence modulo the closed subspace of 
$\Der(\widehat\freeLie,\widehat\freeLie_{\geq 2})$ spanned by Lie brackets, 
 and the first occurence of $\equiv$ follows from the BCH formula. 

(2) The fact that $\IAbel$ is a homomorphism immediately follows from (1). 
The $\Aut(\widehat\freeLie)$-equivariancy of $\IAbel$ is also a consequence of (1) as follows. 
For $\psi \in \IAut(\widehat\freeLie)$ and $\varphi \in \Aut(\widehat\freeLie)$, 
we have 
\begin{align*}
\IAbel (\varphi\psi\varphi^{-1})
&=\Abel (\varphi)+\sigma (\varphi) \cdot \Abel (\psi) + \sigma (\varphi\psi) \cdot \Abel (\varphi^{-1})\\
&=\Abel (\varphi)+\sigma (\varphi) \cdot \Abel (\psi) + \sigma(\varphi) \cdot  \Abel (\varphi^{-1})\\
&=\Abel (\varphi)+\sigma (\varphi) \cdot \Abel (\psi) - \sigma(\varphi) \sigma(\varphi)^{-1} \cdot \Abel (\varphi) 
\ = \ \sigma (\varphi) \cdot \Abel (\psi). 
\end{align*}

\up
\end{proof}

The definition of the map $\Abel$ can be refined as follows.
Let  $\mathfrak{R}\subset \freeLie_{\geq 2}$ be a characteristic ideal of $\freeLie$, and 
recall that $\Aut^{{\mathfrak{R}}}(\widehat{\freeLie})$ is  the subgroup of $\Aut(\widehat{\freeLie})$
acting trivially at the level of $\widehat{\freeLie}/\widehat{\mathfrak{R}}$.
Then, by Lemma~\ref{lem:abelianization}, the map 
$$
\Abel^{\mathfrak{R}}: \Aut^{\mathfrak{R}}(\widehat\freeLie)  \longrightarrow \widehat{H_1}\big(\Der(\widehat\freeLie,\widehat{\mathfrak{R}})\big), \ \psi \longmapsto \log\left(\psi\right)
$$
is a group homomorphism. Clearly, $\Abel^{\mathfrak{R}}=\IAbel$ for $\mathfrak{R}:=\freeLie_{\geq 2}$.

\subsection{The Magnus representation}  \label{subsec:Magnus}

Let  $\mathfrak{R}\subset \freeLie_{\geq 2}$ be a characteristic ideal of $\freeLie$. 
Set $\widehat{T} :=  \widehat{T}(H)$ and $\widehat{T}^{{\mathfrak{R}}}:= \widehat{T}/\widehat{T}{ \mathfrak{R}} \widehat{T}$, 
where  $\widehat{T} \mathfrak{R} \widehat{T}$ denotes the closed two-sided ideal of $\widehat{T}$ generated by $\mathfrak{R}$.

\begin{lemma} \label{lem:Magnus}
Any $\psi \in \Aut^{{\mathfrak{R}}}(\widehat{\freeLie})$ induces an automorphism of the right $\widehat{T}^{{\mathfrak{R}}}$-module
$\widehat{T}_{\geq 1} {\otimes}_{\widehat{T}} \widehat{T}^{{\mathfrak{R}}}$, where $\widehat{T}_{\geq 1}$ is regarded as a right $\widehat T$-module.
\end{lemma}

\begin{proof}
Let $\widetilde{\psi}$ be the extension of $\psi$ to a filtration-preserving algebra automorphism of $\widehat{T}$.
For any $x\in \widehat{T}_{\geq 1}$, $y\in \widehat{T}$ and $u\in \widehat{T}^{\mathfrak{R}}$, we have 
\begin{eqnarray*}
\widetilde{\psi}(xy) \otimes u &= &\widetilde{\psi}(x) \widetilde{\psi}(y) \otimes u \\
& =&   \widetilde{\psi}(x) y \otimes u + \widetilde{\psi}(x)( \widetilde{\psi}(y)-y) \otimes u \ = \   \widetilde{\psi}(x)  \otimes yu  \quad  \in \widehat{T}_{\geq 1} {\otimes}_{\widehat{T}} \widehat{T}^{\mathfrak{R}}. 
\end{eqnarray*}
This computation shows that the endomorphism $\widetilde{\psi} \otimes  \Id_{\widehat{T}^{\mathfrak{R}}} $  of $\widehat{T}_{\geq 1} {\otimes} \widehat{T}^{{\mathfrak{R}}}$ 
induces an endomorphism of $\widehat{T}_{\geq 1} {\otimes}_{\widehat T} \widehat{T}^{{\mathfrak{R}}}$. 
Our claim immediately follows from this.
\end{proof}

The \emph{Magnus representation} of $ \Aut^{{\mathfrak{R}}}(\widehat{\freeLie})$ is the group homomorphism
$$
\Magnus^{{\mathfrak{R}}}:  \Aut^{{\mathfrak{R}}}(\widehat{\freeLie}) \longrightarrow \Aut_{ \widehat{T}^{{\mathfrak{R}}}}\big(\widehat{T}_{\geq 1} {\otimes}_{\widehat T} \widehat{T}^{{\mathfrak{R}}}\big)
$$
resulting from  Lemma \ref{lem:Magnus}. 
For instance, if  $\mathfrak{R}:=\freeLie_{\geq 2}$, 
then the Magnus representation is a group homomorphism  
$$
\IMagnus :  \IAut(\widehat{\freeLie}) \longrightarrow  \Aut_{ \widehat{S}}\big(\widehat{T}_{\geq 1} {\otimes}_{\widehat T}  \widehat{S} \big)
$$
where $ \widehat{S} :=  \widehat{S}(H)$ denotes the degree-completion of the symmetric algebra $S(H)$
generated by the vector space $H$ in degree $1$.

The  homomorphism $\Magnus^{{\mathfrak{R}}}$ has the following concrete description. 
We fix a basis $x:=(x_1,\dots,x_n)$ of $\widehat{T}_{\geq 1}$ as a right $\widehat{T}$-module.
(For instance,  we can pick a basis $x$ of the vector space  $H$, 
and use the inclusion $H \hookrightarrow \widehat{T}_{\geq 1}$.)
Then the right $\widehat{T}^{\mathfrak{R}}$-module  $\widehat{T}_{\geq 1} {\otimes}_{\widehat T} \widehat{T}^{{\mathfrak{R}}}$ is freely generated by $(x_1 \otimes 1,\dots, x_n \otimes 1)$.
By considering matrix presentations relatively to this basis,  we obtain  a group homomorphism 
$$
\Magnus^{{\mathfrak{R}}}_x:  \Aut^{{\mathfrak{R}}}(\widehat{\freeLie}) \longrightarrow \GL\big(n;  \widehat{T}^{\mathfrak{R}}\big)
$$
which depends on the choice of the basis $x$.
Using the Fox derivations $\partial_1,\dots, \partial_n: \widehat{T} \to \widehat{T}$ defined by the identity \eqref{eq:fundamental}, 
we get the ``Jacobian matrix'' formula 
\begin{equation} \label{eq:Jac}
\forall \psi \in  \Aut^{{\mathfrak{R}}}(\widehat{\freeLie}), \quad 
\Magnus^{{\mathfrak{R}}}_x(\psi) = \Big(p^\mathfrak{R}\partial_i\big( \psi(x_j) \big) \Big)_{i,j}
\end{equation}
where $p^{\mathfrak{R}}: \widehat{T} \to  \widehat{T}^{\mathfrak{R}}$ denotes the canonical projection.

\subsection{The  traces correspondence}

Let  $\mathfrak{R}\subset \freeLie_{\geq 2}$ be a characteristic ideal of $\freeLie$. 
By applying Appendix \ref{sec:cyclic_things} to the augmented 
algebra $R:= \widehat{T}^{{\mathfrak{R}}}$ 
and to the free right $R$-module $M:=  \widehat{T}_{\geq 1} {\otimes}_{\widehat T} \widehat{T}^{{\mathfrak{R}}}$, we get a  notion of ``log-determinant''
$$
\ldet  :  \IAut_{ \widehat{T}^{{\mathfrak{R}}}}\big(\widehat{T}_{\geq 1} {\otimes}_{\widehat T} \widehat{T}^{{\mathfrak{R}}}\big)
\longrightarrow \widehat{C^{\mathfrak{R}}}(H).
$$
(Here we are tacitly identifying the degree-completion 
$ \widehat{C^{\mathfrak{R}}}(H)$ of the graded vector space $C^{\mathfrak{R}}(H)= T/([T,T]+ T \mathfrak{R}T)$ 
with $\widehat{T}^{{\mathfrak{R}}}/[\widehat{T}^{{\mathfrak{R}}},\widehat{T}^{{\mathfrak{R}}}]$.)

\begin{lemma} \label{lem:traces}
The following  diagram is commutative:
$$
\xymatrix{
\Aut^{\mathfrak{R}}(\widehat \freeLie) \ar[rr]^-{\Abel^{\mathfrak{R}}} \ar[d]_-{\Magnus^{\mathfrak{R}}}&& 
\widehat{H_1}\big(\Der(\widehat\freeLie,\widehat{\mathfrak{R}})\big) \ar[d]^-{\Trace^{\mathfrak{R}}}\\
 \IAut_{ \widehat{T}^{{\mathfrak{R}}}}\big(\widehat{T}_{\geq 1} {\otimes}_{\widehat T} \widehat{T}^{{\mathfrak{R}}}\big) \ar[rr]_-{\ldet } &  & \widehat{C^{\mathfrak{R}}}(H)
}
$$
\end{lemma}

\begin{proof}
We set $R:= \widehat{T}^{{\mathfrak{R}}}$ and $M:=  \widehat{T}_{\geq 1} {\otimes}_{\widehat T} \widehat{T}^{{\mathfrak{R}}}$.
Choose a basis $x:=(x_1,\dots,x_n)$ of the vector space $H$, 
which  induces a basis $x\otimes 1 := (x_1 \otimes 1,\dots, x_n \otimes 1)$ of the free $R$-module $M$.

Let $\psi \in \Aut^{\mathfrak{R}}(\widehat \freeLie)$ and let $\widetilde{\psi}$ be  the extension of $\psi$ 
to a filtration-preserving algebra automorphism of $\widehat{T}$.  
Then, for any integer $k\geq 1$,
$$
\left(\Magnus^{\mathfrak{R}}(\psi) -\Id_M\right)^k  \in \End_R(M)
$$
is induced by 
$$
\big(\widetilde{\psi}-\Id_{\widehat{T}_{\geq 1}}\big)^k \otimes \Id_R \in \End_R\big( \widehat{T}_{\geq 1} {\otimes} R\big).
$$ 
It follows that  $\log  \Magnus^{\mathfrak{R}}(\psi)  \in \End_R(M)$
is induced by  $\log ( \widetilde{\psi} )\otimes \Id_R$.
Therefore
\begin{eqnarray*}
\ldet\,  \Magnus^{\mathfrak{R}}(\psi) 
 &=&  \trace \big(\log \Magnus^{\mathfrak{R}}(\psi) \big) \\
&\by{eq:trace_as_usual}& \sum_{i=1}^n \left(\begin{array}{c} \hbox{\small $i$-th coordinate of\  $\log ( \widetilde{\psi} )(x_i) \otimes 1$} \\ 
 \hbox{\small in the basis $x\otimes 1$ of the  $R$-module $M$}  \end{array}\right)\\
&=&  \sum_{i=1}^n \left(\begin{array}{c} \hbox{\small $i$-th coordinate of\  $\log ( \widetilde{\psi} )(x_i) $} \\ 
 \hbox{\small in the basis $x$ of the  $\widehat{T}$-module $\widehat{T}_{\geq 1}$}  \end{array}\right) 
 \ = \ \sum_{i=1}^n \partial_i \big(\log ( \widetilde{\psi} )(x_i)\big). 
\end{eqnarray*}
On the other hand, we have
$$
\Trace^{\mathfrak{R}} \Abel^{\mathfrak{R}}(\psi) =  \Trace^{\mathfrak{R}} \log(\psi)
 \by{eq:divergence} \sum_{i=1}^n \partial_i  \big(\log ({\psi} )(x_i)\big).
$$
Since the restriction of $\log ( \widetilde{\psi} ) \in \Der(\widehat T, \widehat T)$ to $\widehat \freeLie$ is obviously equal to $\log(\psi)$,
we conclude that $\ldet\, \Magnus^{\mathfrak{R}}(\psi) =\Trace^{\mathfrak{R}} \Abel^{\mathfrak{R}}(\psi)$.
\end{proof}

\section{Applications to the automorphism groups of free groups}  \label{sec:autFn}

Before applying the  machinery of the  previous sections to groups of homology cobordisms,
we will give an application to the automorphism groups of free groups. 
In this section,  $F$ is a free group of rank $n$.

\subsection{The Andreadakis filtration}

Let $\AutF := \AutF(F)$ denote the automorphism group of $F$.
The study of the group $\AutF$ by means of its action on the successive nilpotent quotients of $F$
started in Andreadakis' work \cite{andreadakis}. It has been further developed by Johnson \cite{johnson_survey} and Morita \cite{morita} in the context of mapping class groups.
(See  Section \ref{subsec:DN} in this connection.) We briefly review his strategy to study $\AutF$.

Recall that $F =  \Gamma_1 F  \supset \Gamma_2 F \supset \cdots \supset \Gamma_k F  \supset  \Gamma_{k+1} F\supset \cdots $
denotes the lower central series of $F$. 
The \emph{Andreadakis filtration} is the descending sequence of subgroups
$$
\AutF = \AutF[0] \supset \AutF[1]   \supset  \cdots \supset \AutF[k]   \supset  \AutF[k+1]    \supset \cdots
$$
where, for any integer $k\geq 1$, we denote by $\AutF[k]$ the subset of all $f\in\AutF$ such that $f(x)x^{-1}\in \Gamma_{k+1} F$ for all $x\in F$.
In particular, $\IAutF := \AutF[1] $ is the subgroup of $\AutF$ acting trivially on the abelianization $H^\Z := F/[F,F]$ of the group $F$. 
It turns out that
$$
\big[ \AutF[k]  , \AutF[\ell]  \big] \subset \AutF[k+\ell], \quad \hbox{for all } k,l\geq 1.
$$
In particular, the sequence $\IAutF  = \AutF[1]   \supset  \AutF[2]   \supset  \AutF[3]   \supset  \cdots$ constitutes a central series of the group $\IAutF$, and it follows that
\begin{equation} \label{eq:inclusion}
\Gamma^k \IAutF \subset \AutF[k] , \quad \hbox{for all } k\geq 1.
\end{equation}
\noindent
It was conjectured by Andreadakis that the inclusion \eqref{eq:inclusion} is actually an equality \cite{andreadakis}. 
Note that \eqref{eq:inclusion} induces a homomorphism of graded Lie rings 
$$
\Upsilon:  \bigoplus_{k=1}^\infty \frac{ \Gamma_k \IAutF}{ \Gamma_{k+1} \IAutF } \longrightarrow \bigoplus_{k=1}^\infty \frac{\AutF[k]}{ \AutF[k+1]  } 
$$
where the Lie bracket on each side is induced by  group commutators in $\AutF$.

As recalled in \eqref{eq:free_free}, the free Lie ring $\freeLie^\Z:= \freeLie^\Z(H^\Z)$ 
can be identified to the graded Lie ring associated to the lower central series of $F$. 
Thus,  the Andreadakis filtration comes with a sequence of homomorphisms 
\[\tau_k: \AutF[k] \longrightarrow \Hom (F, \Gamma_{k+1}F/\Gamma_{k+2}F) \simeq 
\Hom \big(  H^\Z  , \freeLie_{k+1}^\mathbb{Z}\big)\]
indexed by the integers $k\geq 1$, which are nowadays called the {\it Johnson homomorphisms}.
To be more specific, the homomorphism $\tau_k(f)$ is defined for any $f\in \AutF[k]$ by
\[\tau_k (f) (x) = f(x) x^{-1} \in   \Gamma_{k+1}F/\Gamma_{k+2}F, \quad \hbox{for all } x \in F. \] 
Clearly,  $\tau_k$ is a homomorphism with $\Ker \tau_k =\AutF[k+1]$.  
Hence the sequence of all Johnson homomorphisms 
\[\tau: \bigoplus_{k=1}^\infty  \frac{\AutF[k]}{ \AutF[k+1]} \longrightarrow 
 \bigoplus_{k=1}^\infty  \Hom (H^\Z, \freeLie_{k+1}^{\mathbb{Z}}) \simeq
\Der (\freeLie^{\mathbb{Z}}, \freeLie_{\ge 2}^{\mathbb{Z}})\]
is an injective homomorphism of graded Lie rings.
We refer the reader to the article \cite{satoh_survey}, 
which surveys the study of the Andreadakis filtration over the last fifty years. 

\subsection{Comparison with the lower central series}

We show the following, which supports the Andreadakis conjecture affirmatively
and has also been obtained independently by Bartholdi \cite{bartholdi}. (See Remark \ref{rem:Bartholdi} below.)

\begin{theorem}\label{thm:surjectivity}
The natural homomorphism
\[\Upsilon_k \otimes_\Z \Q :(\Gamma_k \IAutF /\Gamma_{k+1} \IAutF) \otimes_\Z \mathbb{Q} 
\longrightarrow (\AutF[k] / \AutF[k+1]) \otimes_\Z \mathbb{Q}\]
is surjective for all $k \ge 1$ and $n \ge k+2$. 
In particular, the graded Lie algebra
\[\bigoplus_{k = 1}^\infty \frac{\AutF [k] }{\AutF [k+1]} \otimes_\Z \mathbb{Q}\]
is stably generated by its degree $1$ part. 
\end{theorem}

A key ingredient in our proof is the following result due to Satoh \cite{satoh}. 

\begin{theorem}[Satoh]\label{thm:satoh}
For any $k \ge 2$ and $n \ge k+2$, there is an exact sequence
\[(\Gamma_k \IAutF /\Gamma_{k+1} \IAutF ) \otimes_\Z  \Q \xrightarrow{\tau'_k} 
\Hom (H, \freeLie_{k+1}) \xrightarrow{\Trace_k } C_k (H) \longrightarrow 0 \]
where $\tau'_k$ denotes the restriction of $\tau_k$ to $\Gamma_k \IAutF$ 
and $\Trace_k$ is the degree $k$ part of the trace cocycle recalled in Section \ref{subsec:trace}.
\end{theorem}

Thus, by Satoh's result, Theorem \ref{thm:surjectivity}  follows from the next statement.

\begin{proposition}\label{prop:ext_satoh}
For any  $n \ge 2$ and  $k\geq 2$,
the map $\Trace_k \circ \tau_k: \AutF [k] \to C_k (H)$ is~zero. 
\end{proposition}

\begin{proof}
We first relate two different notions of ``Fox derivations''.
Fix a basis $\gamma$ of~$F$. 
Let  $\frac{\partial \ }{\partial \gamma_1}, \dots, \frac{\partial \ }{\partial \gamma_n}: \Z[F] \to \Z[F]$ be the Fox derivations relative to this basis. 
They are uniquely defined by the following identity:
$$
x- \varepsilon(x) = \sum_{i=1}^n (\gamma_i  -1)  \frac{\partial x }{\partial \gamma_i}, \quad \hbox{for all } x\in \Z[F].
$$ 
(Note that we are considering the augmentation ideal of $\Z[F]$ as a \emph{right} free $\Z[F]$-module,
whereas the usage is to  consider it as a  \emph{left} free $\Z[F]$-module;
thus the Fox derivations in our sense do not satisfy the  ``Leibniz rules'' and the ``chain rule'' 
to which the reader may be used.) 
Let $\theta:= \theta_\gamma^{\exp}: \pi \to \widehat{T}$ be the group-like expansion defined in Example \ref{ex:expansions}, 
and let $x:=(x_1,\dots,x_n)$ where $x_i:=\theta(\gamma_i-1)$ for all $i\in \{1,\dots,n\}$. Then $x$ is a basis of $\widehat{T}_{\geq 1}$ as a right $\widehat{T}$-module,
and we denote by $\partial_1,\dots,\partial_n$ the corresponding Fox derivations:
$$
y- \varepsilon(y) = \sum_{i=1}^n x_i\,  \partial_i(y), \quad \hbox{for all } y \in \widehat{T}.
$$
Then we get 
\begin{equation} \label{eq:Fox_to_Fox}
\theta\Big(\frac{\partial x }{\partial \gamma_i}\Big)  = \partial_i\big(\theta(x)\big), 
\quad \hbox{for all } x\in \Z[F] \hbox{\ and } i \in \{1,\dots,n\}.
\end{equation}

Let $f\in \AutF [k]$. We can associate to $f$ two  types of  ``Jacobian matrices''. 
The first one is the usual Magnus representation of the group automorphism $f$, 
namely
$$
 \Magnus_\gamma (f) := \Big( \frac{\partial f(\gamma_j) }{\partial \gamma_i} \Big)_{i,j} \ \in \GL(n;\Z[F]).
$$
The second one is an ``infinitesimal version'' of the first one.
Specifically, denote by $\Malcev(f)\in \Aut(\Malcev(F))$ the filtration-preserving automorphism of Lie algebras induced by $f$ in the natural way,
and set $\psi:=  \widehat{\theta} \circ  \Malcev(f) \circ \widehat{\theta}^{-1} \in \Aut(\widehat{\freeLie})$.
Then, the constructions of Section \ref{subsec:Magnus} give 
$$
\Magnus_x^{{\mathfrak{R}}}(\psi) \by{eq:Jac}   \Big(p^\mathfrak{R}\partial_i\big( \psi(x_j) \big) \Big)_{i,j}  \in \GL\big(n; \widehat{T}^{\mathfrak{R}}\big)
$$
where  $\mathfrak{R}:=\freeLie_{\ge k+1}$ and $p^{\mathfrak{R}}: \widehat{T} \to  \widehat{T}^{\mathfrak{R}}$ denotes the canonical projection. 
Let $\widehat{\theta}^{\mathfrak{R}}: \widehat{\Q[F]}  \to \widehat{T}^{\mathfrak{R}}$ 
be the composition of the isomorphism $\widehat{\theta}: \widehat{\Q[F]}  \to \widehat{T}$ with $p^{\mathfrak{R}}$.
It follows from \eqref{eq:Fox_to_Fox} that, regarding $\GL(n;\Z[F])$ as a subset of $\GL(n;\widehat{\Q[F]})$, 
\begin{equation}
\label{eq:Magnus_to_Magnus}
\widehat{\theta}^{  \mathfrak{R}} \big( \Magnus_\gamma (f)
\big) = \Magnus_x^{{\mathfrak{R}}}(\psi).
\end{equation}

Consider now the noncommutative notion of ``log-determinant'' for matrices introduced in Example \ref{ex:ldet_matrices}.
Then we have a commutative diagram
\begin{equation}  \label{eq:ldet_ldet}
\xymatrix{
\GL\big(n;  \widehat{\Q[F]} \big)  \ar[d]_-{\widehat{\theta}^{\mathfrak{R}}} \ar[r]^-\ldet & \widehat{\Q[F]}/ \big[ \widehat{\Q[F]} , \widehat{\Q[F]}\big] \ar[d]^-{\widehat{\theta}^{\mathfrak{R}}}  \\
\GL\big(n;\widehat{T}^{\mathfrak{R}} \big)  \ar[r]^-\ldet &\widehat{C}^{\mathfrak{R}}(H) .
}
\end{equation}
We deduce from Lemma \ref{lem:traces} that
\begin{eqnarray}
\notag \Trace^{\mathfrak{R}} \big(\Abel^{\mathfrak{R}} (\psi) \big)&=& \ldet\big( \Magnus^{{\mathfrak{R}}}(\psi) \big) \\
\notag & = &\ldet \big(\Magnus_x^{{\mathfrak{R}}}(\psi)\big) \\
\label{eq:C} & \by{eq:Magnus_to_Magnus}
& \ldet \, \widehat{\theta}^{  \mathfrak{R}} \big(  \Magnus_\gamma (f) 
\big)  \ \by{eq:ldet_ldet} \  \widehat{\theta}^{  \mathfrak{R}}\, \ldet  \big(  \Magnus_\gamma (f)
\big).
\end{eqnarray}

Next,  we consider the following composition of group homomorphisms: 
\begin{equation} \label{eq:ldet_composed}
\GL(n;  {\Z[F]} ) \hookrightarrow \GL(n;  \widehat{\Q[F]}  ) \stackrel{\ldet\, }{\longrightarrow} \widehat{\Q[F]}/ \big[ \widehat{\Q[F]} , \widehat{\Q[F]}\big].
\end{equation}
According to Example \ref{ex:ldet_matrices}, it factorizes through $K_1({\Z[F]})$.
But, since $F$ is a free group, its Whitehead group is trivial by results of Higman \cite{higman} and Stallings \cite{stallings_Wh}.
Therefore the abelian group $K_1({\Z[F]})$ is generated by the elements $(\pm x)$ for all $x\in F$, 
and we deduce that the image of the group homomorphism \eqref{eq:ldet_composed} is spanned by the classes of the elements $\log(\gamma_i)$ for all $i\in\{1,\dots,n\}$. Note that
$$
 \widehat{\theta}^{  \mathfrak{R}}(\log(\gamma_i))  = [\gamma_i] \in \widehat{C}^{\mathfrak{R}}(H)
$$
is homogeneous of degree $1$. 
Then,  we deduce from \eqref{eq:C} that $\Trace ^{\mathfrak{R}} \Abel^{\mathfrak{R}} (\psi)$ must be  homogeneous of degree $1$.

But, since the group automorphism $f$ gives the identity at the level of $F/\Gamma_{k+1} F$, 
the Lie algebra automorphism $\psi$ induces the identity at the level of $\widehat{\freeLie}/\widehat{\freeLie}_{\geq k+1}$.
Consequently, the derivation $\log(\psi)$ viewed as an element of $\Hom(H, \widehat{\freeLie})$
 belongs to the subspace $\Hom(H, \widehat{\freeLie}_{\geq k+1})$.
Therefore  $\Trace ^{\mathfrak{R}} \Abel^{\mathfrak{R}} (\psi)$ starts in degree $k\geq 2$, so that it must be zero.
Moreover, it follows easily from the definitions that the  leading term of $\log(\psi)$ is $\tau_k(f) \in \Hom(H, \widehat{\freeLie}_{k+1})$,
and we conclude that 
$$
\Trace_k \, \tau_k(f) =\big(\hbox{degree $k$ part of } \Trace ^{\mathfrak{R}} \Abel^{\mathfrak{R}} (\psi)\big)=0.
$$

\up
\end{proof}

\begin{remark} \label{rem:Bartholdi}
It seems that the surjectivity of $\Upsilon_k \otimes_\Z \Q$ is also shown by Bartholdi during the proof of \cite[Theorem C]{bartholdi}.
(His method is rather analogous to our proof of Theorem~\ref{thm:surjectivity}, but with different arguments.)
The authors do not know whether the map $\Upsilon_k \otimes_\Z \Q$ is an isomorphism (for $n$ large enough with respect to $k$).
In fact, it can be verified that the bijectivity of ${\Upsilon_k \otimes_\Z \Q}$ (in a stable range) is equivalent to a weaker form of Andreadakis' conjecture: 
that the Andreadakis filtration coincides (in a stable range)  with the \emph{rational} lower central series of $\IAutF$.
In this connection, see \cite{bartholdi_erratum} for a correction to \cite[Theorem A]{bartholdi}.
\hfill $\blacksquare$
\end{remark}

\section{The abelianization map on the  group of homology cobordisms}\label{sec:h_cob}

In this section, we begin to apply the constructions of the previous sections to the group of homology cobordisms.
Let $\Sigma$ be a compact connected oriented surface of genus $g\geq 1$ with one boundary component.

\subsection{The infinitesimal Dehn--Nielsen representation}  \label{subsec:DN}

Firstly, following \cite{massuyeau_IMH}, 
we review the use of symplectic expansions to define infinitesimal versions  of the Dehn--Nielsen representations, 
and we recall the relationship with Johnson homomorphisms.
The reader may consult the survey article \cite{hm}.

As defined in Section \ref{sec:intro}, $\cob:= \cob(\Sigma)$ is the monoid of homology cobordisms from $\Sigma$ to itself,
and  $\Hcob  := \Hcob(\Sigma)$ is  the quotient of $\cob$  by the relation of  $4$-dimensional homology cobordism.
Let $\pi := \pi_1(\Sigma,\star)$ be the fundamental group of $\Sigma$ based at a point $\star \in \partial \Sigma$, and consider its lower central series
$$
\pi= \Gamma_1 \pi \supset \Gamma_2 \pi \supset \cdots \supset \Gamma_k \pi \supset \Gamma_{k+1} \pi \supset \cdots  . 
$$
Let $k\geq 1$ be an integer and let $M\in \cob$. 
The boundary parametrizations $m_\pm: {\Sigma \to M}$ induce isomorphisms in homology so that, according to Stallings \cite{stallings},
they also induce isomorphisms
$$
m_{\pm,*}: \pi/\Gamma_{k+1} \pi \to \pi_1(M,\star)/ \Gamma_{k+1} \pi_1(M,\star)
$$
at the level of the $k$-th nilpotent quotients.
(Here we are tacitly identifying the points $m_+(\star)$ and $m_-(\star)$ to a single point $\star$ 
in the interior of the vertical boundary of $M$.) 
Thus we can consider the group automorphism
$$
\rho_k(M) :=   (m_{-,*})^{-1} \circ m_{+,*} \in \Aut(\pi/\Gamma_{k+1} \pi),
$$
which only depends on the $4$-dimensional homology cobordism class of $M$. This defines a group homomorphism 
$$
\rho_k: \Hcob \longrightarrow  \Aut(\pi/\Gamma_{k+1} \pi), 
$$
which induces an action on the Malcev Lie algebra $\Malcev(\pi/\Gamma_{k+1}\pi)$ 
compatible with the canonical homomorphisms $\Malcev(\pi/\Gamma_{k+1}\pi) \to \Malcev(\pi/\Gamma_{\ell+1}\pi)$ for any  $k\geq \ell$. 
Therefore, by passing to the inverse limit, we obtain an action
$$
\varrho: \Hcob   \longrightarrow  \Aut\big(\Malcev(\pi)\big)
$$
of $\Hcob$ on the complete Lie algebra $\Malcev(\pi)$
or, equivalently, an action of $ \Hcob $ on the complete Hopf algebra $\widehat{\Q[\pi]}$.

One can think of $\varrho$ as an ``infinitesimal'' version of the Dehn--Nielsen representation. 
Usually, the \emph{Dehn--Nielsen representation} refers to the  canonical action $\rho: \mcg \to \Aut(\pi)$ 
of the mapping class group $\mcg := \mcg(\Sigma)$ on the fundamental group $\pi$ of $\Sigma$.
The mapping cylinder construction defines a map $\mathsf{c}: \mcg \to \Hcob$, and it is easily checked that we have a commutative diagram
$$
\xymatrix{
\mcg \ar[d]_-{\mathsf{c}} \ar[r]^-\rho & \Aut(\pi)  \ar[d]^-{\Malcev} \\
\Hcob \ar[r]_-\varrho & \Aut(\Malcev(\pi)).
}
$$
Since $\pi$ is a free group, the canonical map $\pi \to \widehat{\Q[\pi]}$ is injective 
so that the canonical homomorphism $\Malcev: \Aut(\pi) \to   \Aut(\Malcev(\pi))$ is injective.
A classical result of Dehn and Nielsen asserts that $\rho$ is injective, and it follows that $\mathsf{c}$ is 
injective too~\cite{gl}. 
In the sequel, we will omit the notation $\mathsf{c}$ and we will simply regard $\mcg$ as a subgroup of~$\Hcob$.

Let now  $\theta:  \pi \to \widehat{T}(H)$ be a symplectic expansion in the sense of Section \ref{subsec:symp_expansion}, where $H:= H_1(\Sigma;\Q)$.
Then, for any $M\in \Hcob$, the automorphism $\varrho(M)$ fixes $\log(\zeta)$ where  $\zeta:= [\partial \Sigma] \in \pi$, 
so that the automorphism $\varrho^\theta(M) := \widehat \theta \circ \varrho(M) \circ \widehat \theta^{-1}$ of the complete free Lie algebra $\widehat{\freeLie} = \widehat{\freeLie}(H)$
fixes the bitensor $\omega$ corresponding the homology intersection form. Hence we get a group homomorphism
$$
\varrho^\theta: \Hcob   \longrightarrow  \Aut_\omega\!\big(\widehat{\freeLie} \big).
$$

The \emph{Johnson filtration} of  $\Hcob$ is the descending sequence of subgroups 
$$
\Hcob = \Hcob[0] \supset \Hcob[1] \supset \cdots \supset \Hcob[k] \supset \Hcob[k+1] \supset \cdots
$$
that is defined by $\Hcob[k]:=\Ker \rho_{k+1}$ for any integer $k\geq 0$. 
In particular, $\Hcyl := \Hcob[1]$ is the \emph{group of homology cylinders} over $\Sigma$ or, equivalently,
it  is the kernel of the group homomorphism
$$
\sigma  : \Hcob \longrightarrow \Sp(H)
$$
that assigns to any $M\in \Hcob$ the linear map   $(m_{-,*})^{-1} \circ m_{+,*}$ 
where $m_{\pm,*}$ is the map induced by $m_\pm$ in homology.  (Note that $\sigma  =  \rho_1$.) 
By restricting to mapping cylinders, we obtain the \emph{Johnson filtration}
$$
\mcg = \mcg[0] \supset \mcg[1] \supset \cdots \supset \mcg[k] \supset \mcg[k+1] \supset \cdots
$$
of the mapping class group, whose first term $\Torelli := \mcg[1]$ is the \emph{Torelli group} of $\Sigma$. 
If we regard $\Torelli$ as a subgroup of $\Aut(\pi)$ via the Dehn--Nielsen representation, 
the Johnson filtration corresponds to the Andreadakis filtration introduced in Section~\ref{sec:autFn}.

For any $M\in \Hcyl$, we have $\log \varrho^\theta(M) \in \Der_\omega(\widehat{\freeLie}, \widehat{\freeLie}_{\geq 2}) = \widehat{\mathfrak{h}}^+$ 
and, for any integer $k\geq 1$, 
 $M$ belongs to $\Hcob[k]$ if and only if  $\log \varrho^\theta(M)$ starts in degree $k$. The \emph{$k$-th Johnson homomorphism} is the group homomorphism
$$
\tau_k: \Hcob[k] \longrightarrow \mathfrak{h}_k
$$
that associates to any $M \in \Hcob[k]$ the leading term of $\log \varrho^\theta(M)$. 
Recall that $\mathfrak{h}_k$ denotes the subspace of $ \widehat{\mathfrak{h}}$ consisting of derivations of $\widehat{\freeLie}$  increasing degrees by $k$
(and vanishing on $\omega$): thus  $\mathfrak{h}_k$  can be viewed as a subspace of $\Hom(H, \freeLie_{k+1})$.
Actually, it is easy to formulate $\tau_k$ only in terms of  $\rho_{k+1}:  \Hcob \to  \Aut(\pi/\Gamma_{k+2} \pi)$.
This shows that $\tau_k$ takes values in 
$$
\mathfrak{h}_k^\Z\subset \Hom(H^\Z, \freeLie^\Z_{k+1})
$$
where $H^\Z:= H_1(\Sigma;\Z)$ and $ \freeLie^\Z$ is the free Lie ring generated by $H^\Z$ in degree $1$;
in particular, the restriction of $\tau_k$ to $\mcg[k]$  coincides with the $k$-th Johnson homomorphism introduced in Section \ref{sec:autFn}. 

The Johnson homomorphisms  have been originally introduced by Johnson himself for  subgroups of
the mapping class groups \cite{johnson,johnson_survey},
 and they have been henceforth studied by Morita \cite{morita}.
Their extension to the group of homology cobordisms, which is already evoked   in \cite{habiro}, 
has been  introduced by Garoufalidis \& Levine \cite{gl}.
In particular, they proved using surgery techniques and cobordism theory that $\tau_k: \Hcob[k] \to \mathfrak{h}_k^\Z$ is surjective. 
(See also \cite{turaev} for a similar result and \cite{habegger} for an independent proof of this important fact.)

\subsection{The abelianization map} \label{subsec:abel}

Let $\theta$ be a symplectic expansion of~$\pi$. We consider the map $\Abel^\theta: \Hcob \to \widehat{H_1} (\widehat{\mathfrak{h}}^+)$ defined by 
$$
\Abel^\theta(M) := \log\big(\varrho^\theta(M) \sigma(M)^{-1}\big)
$$
for any $M \in \Hcob$. 
Here $\varrho^\theta:  \Hcob \to \Aut_\omega(\widehat{\freeLie})$  and  $\sigma : \Hcob \to  \Sp(H)$ 
 are the group homomorphisms introduced in Section \ref{subsec:DN}, 
and $\Sp(H)$ is embedded in $\Aut_\omega(\widehat{\freeLie})$ in the canonical way.

\begin{theorem}\label{thm:main}
$(1)$ The map $\Abel^\theta$ is a group $1$-cocycle: for all   $M, N \in \Hcob$, we have 
$$
\Abel^\theta(MN) = \Abel^\theta(M) + \sigma(M)\cdot \Abel^\theta(N).
$$
$(2)$ The restriction $\IAbel$ of $\Abel^\theta$ to $\Hcyl$ does not depend on $\theta$.

\noindent
$(3)$ The map $\IAbel: \Hcyl \to \widehat{H_1} (\widehat{\mathfrak{h}}^+)$ is an $\Hcob$-equivariant group homomorphism.

\noindent
$(4)$ For any $d\geq 1$, the image of $\IAbel$ truncated at degree $\leq d$ spans the $\Q$-vector space  $\widehat{H_1}(\widehat{\mathfrak{h}}^+)/\widehat{H_1}(\widehat{\mathfrak{h}}^+)_{\geq d+1}$.

\end{theorem}

\begin{proof}
The map $\Abel^\theta$ is the composition of the group homomorphism $\varrho^\theta: \Hcob \to \Aut_\omega(\widehat{\freeLie})$
with  the  map $\Abel: \Aut(\widehat\freeLie)  \to  \widehat{H_1}\big(\Der(\widehat\freeLie,\widehat \freeLie_{\geq 2})\big)$ 
that has been considered in Section \ref{subsec:abelianization}. Thus the statements (1) and (3) follow from Lemma \ref{lem:abelianization}.

We now prove (2). Let $\theta'$ be another symplectic expansion of $\pi$:
there exists $\psi \in \IAut_\omega(\widehat{\freeLie})$ such that 
$\theta'= \psi \circ \theta$. Then, for any $M \in \Hcob$, we have
\begin{eqnarray*}
\varrho^{\theta'}(M) \sigma(M)^{-1}
& = & \psi \varrho^\theta(M) \psi^{-1} \sigma(M)^{-1} \\
& = & \psi   \big(\varrho^{\theta}(M) \sigma(M)^{-1}\big)  \big(\sigma(M) \psi^{-1} \sigma(M)^{-1}\big);
\end{eqnarray*}
hence
\begin{eqnarray}
\Abel^{\theta'} (M)
\notag & \equiv & \log(\psi) + \Abel^{\theta} (M) + \log (\sigma(M) \psi^{-1} \sigma(M)^{-1})\\
\notag &= &\log(\psi) + \Abel^{\theta} (M) +  \sigma(M) \log(\psi^{-1})  \sigma(M)^{-1} \\
\label{eq:for_further_use} &=& \Abel^{\theta} (M) + \log(\psi) - \sigma(M) \cdot \log(\psi)
\end{eqnarray}
where $\equiv$ denotes a congruence modulo the closed subspace spanned 
by Lie brackets in $\widehat{\mathfrak{h}}^+$.
In particular, for any $M \in \Hcyl$, we obtain $\Abel^{\theta'} (M) = \Abel^{\theta} (M)$.

Finally, we prove (4) in the following equivalent form:  
the $\Q$-vector space span\-ned by the image of $\IAbel$ is dense in $\widehat{H_1}(\widehat{\mathfrak{h}}^+)$ 
with respect to the degree topology. 
Let $R\subset \widehat{\mathfrak{h}}^+$ be the $\Q$-vector space spanned by the image of $\log \varrho^\theta$:
it suffices to show that $R$ is dense in  $\widehat{\mathfrak{h}}^+$.
For a given element $x$ in $\widehat{\mathfrak{h}}^+$, we shall construct a sequence $(y_n)_{n\geq 1}$ in $\widehat{\mathfrak{h}}^+$
such that $y_n$ belongs to  $\widehat{\mathfrak{h}}^+_{\geq n}\cap R$ for any $n\geq 1$ and $x-\sum_{n=1}^N y_n$  belongs to $\widehat{\mathfrak{h}}^+_{\geq N+1}$ for any $N \geq 1$:
it will follow that $x= \sum_{n=1}^\infty y_n$ is the limit of a sequence of elements of $R$.
We  proceed inductively as follows. 
Assume that $y_1, \dots, y_{N-1}$ have been constructed with the required properties.
Let $u_{N} \in \mathfrak{h}_{N}$ be the degree $N$ part of $x-\sum_{n=1}^{N-1} y_n$ 
and let $z_{N} \in \Z\setminus\{0\}$ be such that $z_{N} u_{N}$ belongs  to $\mathfrak{h}_{N}^\Z$. 
By the result of  Garoufalidis \& Levine  that has been recalled at the end of Section \ref{subsec:DN},
there exists a $C_N \in \Hcob[N]$ such that $\tau_N(C_N)= z_{N} u_{N}$. 
Hence $y_N:= \frac{1}{z_N} \log\varrho^\theta(C_N)$ belongs to $\widehat{\mathfrak{h}}^+_{\geq N}\cap R$, and 
$$
x- \sum_{n=1}^N y_n =  \Big(x- \sum_{n=1}^{N-1} y_n\Big) - y_N = u_N  +  (\deg \geq N+1)- y_N
$$ 
belongs to $\widehat{\mathfrak{h}}^+_{\geq N+1}$, which shows the inductive step.
\end{proof}

We now explain how the $1$-cocycle $\Abel^\theta: \Hcob \to \widehat{H_1}(\widehat{\mathfrak{h}}^+)$ relates to the constructions of \cite{morita_GD} for $g\geq 2$.
First, we regard it  as a group homomorphism
$$
\widetilde\Abel^\theta : \Hcob \longrightarrow \widehat{H_1} (\widehat{\mathfrak{h}}^+) \rtimes \Sp(H). 
$$
Next, by composing with the projection $p:\widehat{H_1} (\widehat{\mathfrak{h}}^+)  \to  \mathfrak{h}_1  = \wedge^3 H$ 
and with the  map $\ITrace: \widehat{H_1} (\widehat{\mathfrak{h}}^+)  \to \widehat{S}(H)$ in degrees greater than $1$, we obtain a group homomorphism 
$$
\mathsf{Mor}^\theta := \big((p \oplus \ITrace )\rtimes \operatorname{Id} \big) \circ  \widetilde\Abel^\theta : 
\Hcob \longrightarrow \Big(\wedge^3 H \oplus \prod_{k=1}^\infty S^{2k+1}( H ) \Big) \rtimes \Sp(H).
$$ 
(According to \eqref{eq:nullity}, we can ignore the even part of $\widehat{S}(H)$.) 
For any integer $d \geq 0$, the image of $\mathsf{Mor}^\theta$ truncated at degree $\leq (2d+1)$  is Zariski-dense in
$$
\Big(\wedge^3 H \oplus \bigoplus_{k=1}^{d}  S^{2k+1}( H ) \Big) \rtimes \Sp(H).
$$
(This follows from  \eqref{eq:surjectivity}, the last statement of Theorem \ref{thm:main} and 
the fact that $\Sp(2g;\Z)$ is Zariski-dense  in $\Sp(2g;\Q)$.) 
Thus, the homomorphism $\mathsf{Mor}^\theta$ enjoys the same properties as  the homomorphism   constructed in \cite[Theorem 5.1]{morita_GD}. 

\begin{remark}
Morita's construction in \cite{morita_GD} needs iterated extensions of nilpotent groups and is based on the theory of algebraic groups.
Our definition of $\mathsf{Mor}^\theta$ seems to be more  direct, and easier to study too.
For instance, it follows directly from Theorem~\ref{thm:main} (2) \& (3) that its restriction
$$
\mathsf{Mor} : \Hcyl \longrightarrow \wedge^3 H \oplus \prod_{k=1}^\infty S^{2k+1}( H ) 
$$
to the group of homology cylinders is canonical and $\Hcob$-equivariant.

Conant, Kassabov and Vogtmann  showed in \cite{CKV} that 
there exist many pieces of $H_1 (\mathfrak{h}^+)$ other than 
Johnson's component  $\wedge^3 H$ and 
Morita's trace components $\bigoplus_{k=1}^\infty S^{2k+1}( H )$.  
Thus our $1$-cocycle $\Abel^\theta$ gives actually a larger target than Morita's original one. \hfill $\blacksquare$
\end{remark}

\subsection{Relation with the LMO homomorphism}

We now explain how the homomorphism $\IAbel: \Hcyl \to \widehat{H_1} (\widehat{\mathfrak{h}}^+)$ is related to the theory of finite-type invariants.

Let $\cyl$ be the submonoid of $\mathcal{C}$ that acts trivially on $H_1(\Sigma;\Z)$:
the quotient of $\cyl$ by the $4$-dimensional relation of homology cobordism is the group $\Hcyl$ of homology cylinders defined in Section~\ref{subsec:DN}.
The \emph{LMO homomorphism} is a multiplicative~map
$$
Z: \cyl \longrightarrow \widehat{\A}(H)
$$ 
with values in the  algebra $\widehat{\A}(H)$ of \emph{symplectic Jacobi diagrams}. 
We will not recall the definitions here, refering to   \cite{hm} for a survey and to \cite{hm_SJD} for further details.
In fact, we only need the ``tree-reduction'' $Z^t$ of $Z$ which we outline below.

The closed subspace of $\widehat{\A}(H)$ spanned by looped symplectic Jacobi diagrams is an ideal: 
thus we can consider the quotient algebra, which is denoted by $\widehat{\A}^t(H)$. 
As a vector space, ${\A}^t(H)$ can  be defined by generators and relations as follows. 
The generators of ${\A}^t(H)$ are tree-shaped unitrivalent finite graphs, 
whose univalent vertices are colored by $H$ and whose trivalent vertices are oriented. 
(We exclude trees having a connected component without trivalent vertex.) 
The relations are the ``multilinearity'' relation at  the $H$-colored univalent vertices, and the ``AS'', ``IHX'' relations shown below: 
\begin{center}
\labellist \small \hair 2pt
\pinlabel {AS} [t] at 102 -5
\pinlabel {IHX} [t] at 543 -5
\pinlabel {$= \ -$}  at 102 46
\pinlabel {$-$} at 484 46
\pinlabel {$+$} at 606 46
\pinlabel {$=0$} at 721 46 
\endlabellist
\centering
\includegraphics[scale=0.4]{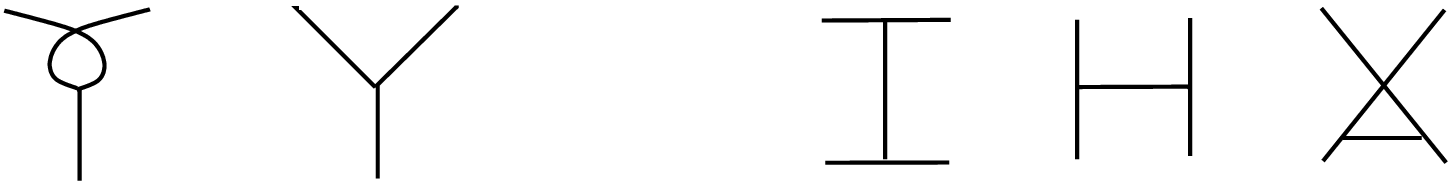}
\end{center}
\vspace{0.5cm}
(In the figures, the orientation at each trivalent vertex of a  tree is assumed to be counter-clockwise.) 
The {\em degree} of a tree  being defined as the number of its trivalent vertices, 
we obtain a grading on the space ${\A}^t(H)$,
and $\widehat{\A}^t(H)$ is the degree-completion of ${\A}^t(H)$. 

The subspace ${\A}^{t,c}(H)$ of ${\A}^t(H)$ spanned by connected tree diagrams is known to be isomorphic 
to the  graded space $\mathfrak{h}^+$ via the degree-preserving map
\begin{equation}  \label{eq:eta}
\eta: \A^{t,c}(H) \ {\longrightarrow}\  \mathfrak{h}^+ \subset H \otimes  \freeLie_{\geq 2},
\quad T \longmapsto \sum_v \operatorname{col}(v) \otimes T_v.
\end{equation}
Here, the sum is over all univalent vertices $v$ of $T$: $ \operatorname{col}(v)\in H$ denotes its color,
and $T_v$ is the Lie word in $H$ that can be read from $T$ when the latter is ``rooted'' at~$v$.
The algebra structure of ${\A}^t(H)$, which we have evoked above, 
induces a Lie bracket on $ \A^{t,c}(H)$ corresponding to the Lie bracket of derivations via $\eta$. 

The composition of $Z$ with the canonical projection $\widehat{\A}(H) \to \widehat{\A}^t(H)$ is denoted~by
$$
Z^t: \cyl \longrightarrow \widehat{\A}^t(H).
$$ 
The closed subspace of $\widehat{\A}^t(H)$ spanned by connected tree diagrams is denoted by $\widehat{\A}^{t,c}(H)$.
The algebra $\widehat{\A}^t(H)$ is actually a complete Hopf algebra whose primitive part is $\widehat{\A}^{t,c}(H)$, 
and the map $Z^t$ takes values in the group-like part of $\widehat{\A}^{t}(H)$.
It is proved in \cite{massuyeau_IMH} that there is a  symplectic expansion $\theta$ such that the following diagram is commutative:
\begin{equation}   \label{eq:Z^t}
\xymatrix @R=1em {
 & \operatorname{GLike} \widehat{\A}^t(H) \ar[r]^-{\log}_-\simeq &  \widehat{\A}^{t,c}(H) \ar[dd]_-\simeq^-{\eta} \\
\cyl \ar[ru]^-{Z^t} \ar[rd]_-{\varrho^{\theta}} & & \\
& \IAut_\omega(\widehat{\freeLie}) \ar[r]_-{\log}^-\simeq  & \widehat{\mathfrak{h}}^+;
}
\end{equation}
in particular, $Z^t$ factorizes through $\Hcyl$. 
(This ``preferred'' symplectic expansion $\theta$ is constructed using the LMO functor \cite{chm}, from which $Z$ originates.) Let 
 $$
 z^t:\Hcyl \longrightarrow \widehat{H_1}\big(  \widehat{\A}^{t,c}(H)\big)
 $$
be the composition of $\log \circ Z^t$ with the canonical map $  \widehat{\A}^{t,c}(H) \to \widehat{H_1}\big(  \widehat{\A}^{t,c}(H)\big)$.
The following  is a direct consequence of  \eqref{eq:Z^t}.

\begin{proposition} \label{prop:Z^t}
We have a commutative diagram:
$$
\xymatrix @R=0.75em  {
&  \widehat{H_1}\big(  \widehat{\A}^{t,c}(H)\big) \ar[dd]^-{\eta_*}_\simeq   \\
\Hcyl \ar[ru]^-{z^t} \ar[rd]_-\IAbel  & \\
& \widehat{H_1} \big( \widehat{\mathfrak{h}}^+ \big)
}
$$
\end{proposition}

This proposition has several consequences which we now discuss. 
In one direction,  since $\IAbel$ is a canonical homomorphism (by statement (2) of Theorem \ref{thm:main}),
it shows that $z^t$ does not depend on the choices upon which the construction of $Z$ is based
(namely a system of meridians \& parallels on $\Sigma$, and a Drinfeld associator).
In the other direction, we deduce two further properties for $\IAbel$.

The first property is about stabilizations.
Let $\Sigma'$ be the boundary-connected sum of $\Sigma$ 
with a compact connected oriented surface of genus $1$ having a single boundary component.
The notations $\Hcob, H,\dots$  used for $\Sigma$ have their exact analogues $\Hcob', H',\dots$ for $\Sigma'$.
Note that the inclusion $\Sigma \hookrightarrow \Sigma'$ induces some maps $\Hcob \to \Hcob'$, $H\to H',\dots$

\begin{cor} \label{cor:stabilization}
The following diagram is commutative:
$$
\xymatrix{
 \Hcyl \ar[r]^-\IAbel  \ar[d] & \widehat{H_1} (\widehat{\mathfrak{h}}^+) \ar[d] \\
  \Hcyl' \ar[r]_-\IAbel & \widehat{H_1} (\widehat{\mathfrak{h}}^{'+})
}
$$
\end{cor}

\begin{proof}
It is easily seen from its construction that 
the LMO homomorphism $Z$ commutes with stabilizations of the reference  surface. 
Indeed, $Z$ is constructed from the LMO functor, which has the property to preserve some monoidal structures~\cite{chm}. 
Hence the homomorphism $z^t:\Hcyl \to \widehat{H_1}\big(  \widehat{\A}^{t,c}(H)\big)$ 
enjoys the same property.
We conclude thanks to Proposition \ref{prop:Z^t}.
\end{proof}

Here is the second property that we deduce for $\IAbel$.

\begin{cor} \label{cor:fti}
 For any integer $d\geq 1$, the truncation of $\IAbel$ at degree $\leq d$  is a finite-type invariant of degree $\leq d$.
\end{cor}

\begin{proof}
It is known that the truncation of $Z$ at degree $\leq d$  is a finite-type invariant of degree~$\leq d$: see \cite{chm,hm_SJD}. 
Thus the corollary follows directly from Proposition~\ref{prop:Z^t}.
\end{proof}

\begin{remark} \label{rem:claspers}
The proof of Corollary \ref{cor:fti} also shows the following, which refines statement (4) in Theorem \ref{thm:main}.
Let $x$ be any homogeneous element of   $\widehat{H_1} \big( \widehat{\mathfrak{h}}^+ \big)$ of degree $d\geq 1$.
Then, using Habiro's terminology \cite{habiro}, 
there is a non-zero integer $n$ and some  connected ``graph claspers''  $G_1,\dots, G_k \subset \Sigma \times [0,1]$, each having $d$ ``nodes'', such that
$$
\IAbel\big( (\Sigma \times [0,1])_{G_1}\big) + \cdots + \IAbel\big( (\Sigma \times [0,1])_{G_k}\big) = nx
$$
where $(\Sigma \times [0,1])_{G_i}$ denotes the cobordism obtained from the trivial cobordism by ``surgery'' along $G_i$.
This is proved as follows: we consider a linear combination $c$ of tree diagrams such that $\eta(c)$ represents $x$,
we multiply $c$ by an integer to have its coefficients in $\Z$,
and we realize each of these symplectic Jacobi diagrams by a (tree-shaped) graph clasper in  $\Sigma \times [0,1]$. 
Here we are using the universality of $Z$ among finite-type invariants. \hfill $\blacksquare$
\end{remark}

\section{The Magnus representation of the group of homology cobordisms}\label{sec:trace_h_cob}

Let $\Sigma$ be a compact connected oriented surface of genus $g\geq 1$ with one boundary component.
We continue our study of the group $\Hcob = \Hcob(\Sigma)$ of homology cobordisms.

\subsection{The classical Magnus representation} \label{subsec:classical_Magnus}

Let $(R,\varepsilon)$ be a commutative augmented algebra, whose group of invertible elements is denoted by~$R^\times$.
We assume that $\varepsilon^{-1}(1) \subset R^\times$.

For any pair of CW-complexes $(X,Z)$, and any group homomorphism $\psi:H_1(X)\to R^\times$, 
the \emph{homology} of $(X,Z)$ with \emph{$\psi$-twisted} coefficients in $R$ is
$$
H^\psi_*(X,Z  ; R  ) := H_* \big( R \otimes_{\Z[\pi_1(X)]} C(\widetilde{X}, p^{-1}_X(Z))\big)
$$
where $p_X:\widetilde{X}\to X$ is the universal cover of $X$, the cellular decomposition of $(X,Z)$ is lifted to the pair $(\widetilde{X}, p^{-1}_X(Z))$
and $C(\widetilde{X}, p^{-1}_X(Z))$ is the resulting cell chain complex with the structure of left $\Z[\pi_1(X)]$-module given by deck transformations.
Note that $H^\psi_*(X,Z;R)$ has a structure of left $R$-module.

\begin{lemma} \label{lem:KLW}
Let $(X,Y,Z)$ be a triple of CW-complexes such that $i_*:H_*(Y;\Z) \to  H_*(X;\Z)$ is  an isomorphism,  where $i:Y \to X$ is the inclusion map. 
Let $\psi: H_1(X) \to R^\times$ be a group homomorphism whose image is contained in $\varepsilon^{-1}(1)$. Then the $R$-linear map
$$
j_*: H_*^{\psi\circ i_*} (Y,Z;R) \longrightarrow H_*^{\psi} (X,Z;R),
$$
which is induced by the inclusion $j:(Y,Z)\to (X,Z)$,  is an isomorphism.
\end{lemma}

\begin{proof}
The lemma is  a variation of \cite[Proposition 2.1]{klw} and it can be proved with similar arguments.
\end{proof}

We now come back to the surface $\Sigma$ and  we fix a group homomorphism $\varphi: H_1(\Sigma;\Z) \to R^\times$ whose image is contained in $\varepsilon^{-1}(1)$. 
Let $\cob^\varphi := \cob^\varphi( \Sigma )$ be the submonoid of $\cob(\Sigma)$ 
consisting of those cobordisms $M$ such that $\varphi\circ (m_{+,*})^{-1} =  \varphi\circ (m_{-,*})^{-1}: {H_1(M;\Z) \to R^\times}$.
The image of $\cob^\varphi$ by the canonical projection $\cob \to \Hcob$ is denoted by $\Hcob^\varphi$.

For any $M \in \cob^\varphi$, the fact that $m_{\pm,*}: H_1(\Sigma;\Z) \to H_1(M;\Z)$ is an isomorphism implies by Lemma \ref{lem:KLW} that the $R$-linear map  
$$ 
n_{\pm,*}:H_1^\varphi(\Sigma,\star\, ;R) \longrightarrow H_1^{\varphi\circ (m_{\pm,*})^{-1}}(M,\star\, ;R)
$$ 
induced by the inclusion $n_\pm: (\Sigma,\star) \to (M,\star)$ is an isomorphism. 
Therefore we can consider the $R$-linear automorphism 
$
r^\varphi(M):=(n_{-,*})^{-1}\circ n_{+,*} 
$
of $H_1^\varphi(\Sigma,\star\, ;R)$,  which only depends on the $4$-dimensional homology cobordism class of~$M$. 
Thus we get a group homomorphism
$$
r^\varphi: \Hcob^\varphi \longrightarrow \Aut_R\big( H_1^\varphi(\Sigma,\star\, ;R)\big),
$$
which we call the \emph{Magnus representation}.
See \cite{sakasai_survey} for a survey of this invariant of homology cobordisms.

In the sequel, we are interested in the following case: $R:=\widehat{S}(H)$ is 
the degree-completion of the symmetric algebra $S(H)$ generated by $H$ in degree $1$,
and $\varphi: H_1(\Sigma;\Z) \to \widehat{S}(H)^\times$  is  the \emph{exponential} map defined by
$$
\varphi(h) := \exp(h) = \sum_{k=0}^\infty \frac{h^k}{k!}
$$
for any $h\in H_1(\Sigma;\Z)$.
Since $\varphi$ is injective in this case, $\Hcob^\varphi$ coincides with the group of homology cylinders $\Hcyl$.
Then the Magnus representation is  a group homomorphism 
$$
r: \Hcyl \longrightarrow \Aut_{\widehat{S}} \big( H_1\big(\Sigma,\star\, ; \widehat{S}\big)\big)
$$
where $\widehat{S}:= \widehat{S}(H)$ and  $ H_1(\Sigma,\star\, ;\widehat{S})$ is 
the first homology group of $(\Sigma,\star)$ with twisted coefficients in $\widehat{S}$.

\subsection{The traces correspondence}

We now relate the determinant of the Magnus representation of $\Hcyl$ to the abelianization map. 

\begin{theorem} \label{thm:Magnus_Morita}
The following diagram is commutative: 
$$
\xymatrix  @!0 @R=0.8cm @C=5cm  {
& \widehat{H_1} (\widehat{\mathfrak{h}}^+) \ar[r]^-{\ITrace} & \prod_{k=1}^\infty {S}^k(H) \ar[dd]^-\exp \\ 
\Hcyl \ar[ru]^-{\IAbel} \ar[rd]_-{r}  &  & \\
& \Aut_{\widehat{S}} \big( H_1\big(\Sigma,\star\, ; \widehat{S}\big)\big) \ar[r]^-\det &  \widehat{S}(H) 
}
$$
\end{theorem}

\begin{proof} 
We first relate the Magnus representation $r$ of $\Hcyl$ 
to the Magnus representation $\IMagnus$ of $\IAut(\widehat{\freeLie})$ introduced in Section \ref{subsec:Magnus}.
For this, we choose a symplectic expansion $\theta$ of $\pi$. 
We claim that there is an automorphism of $\widehat{S}$-modules
$$
\psi^\theta: H_1\big(\Sigma,\star\, ; \widehat{S}\big) \stackrel{\simeq}{\longrightarrow} \widehat{T}_{\geq 1} \otimes_{\widehat{T}} \widehat{S}
$$
depending on $\theta$, and such that the following diagram commutes for any $M\in \Hcyl$: 
\begin{equation} \label{eq:psi's}
\xymatrix @!0 @R=1.3cm @C=2cm{
H_1\big(\Sigma,\star\, ; \widehat{S}\big) \ar[d]_-{\psi^\theta}^-{\simeq}  \ar[rr]^-{r(M)} && H_1\big(\Sigma,\star\, ; \widehat{S}\big)  \ar[d]^-{\psi^\theta}_-{\simeq} \\
 \widehat{T}_{\geq 1} \otimes_{\widehat{T}} \widehat{S} \ar[rr]_-{\IMagnus (\varrho^\theta(M) ) } &&  \widehat{T}_{\geq 1} \otimes_{\widehat{T}} \widehat{S}
}
\end{equation}
Specifically, we shall define $\psi^\theta$ as a composition of several isomorphisms 
\begin{equation} \label{eq:four_isos}
H_1\big(\Sigma,\star\, ; \widehat{S}\big) \simeq \widehat{S} \otimes_{\Q[\pi]} I \simeq \widehat{S} \otimes_{\widehat{\Q[\pi]}}\widehat{I} 
\simeq  \widehat{I} \otimes_{\widehat{\Q[\pi]}} \widehat{S}  \simeq  \widehat{T}_{\geq 1} \otimes_{\widehat{T}} \widehat{S}
\end{equation}
where $I$ denotes the augmentation ideal of the group algebra  $\Q[\pi]$. 

 Consider a cell decomposition of $\Sigma$ with  $\star$ as a unique $0$-cell. 
 The first isomorphism in \eqref{eq:four_isos} is the  $\widehat{S}$-linear map 
$H_1\big(\Sigma,\star\, ; \widehat{S}\big) \to \widehat{S} \otimes_{\Q[\pi]} {I}$ defined by 
\begin{equation}   \label{eq:natural_map}  
\Big[ \sum_{j \in J} r_j \otimes e_j \Big] \longmapsto \sum_{j \in J} r_j \otimes \big(e_j(1)-e_j(0)\big)
\end{equation}
where we use the following notations: $J$ is a finite set; $r_j \in \widehat{S}$ and 
$e_j \subset \widetilde{\Sigma}$ is a lift of an oriented $1$-cell of $\Sigma$ for all $j\in J$; 
$e_j(1), e_j(0) \in \pi$ are such that $\partial_1(e_j) =  (e_j(1)-e_j(0))\, \widetilde{\star}$ where $\widetilde{\star}$ 
denotes the preferred lift of $\star$ to $\widetilde{\Sigma}$.
(That the natural map \eqref{eq:natural_map} is an isomorphism follows from the fact that $\Sigma$  
deformation retracts to a wedge of circles based at~$\star$.)
The second isomorphism in \eqref{eq:four_isos} is the $\widehat{S}$-linear map induced by the inclusion $I \hookrightarrow \widehat{I}$.
(It is an isomorphism since, given a basis $(\gamma_1,\dots,\gamma_{2g})$ of $\pi$, 
the $\Q[\pi]$-basis  $(\gamma_1 -1 ,\dots,\gamma_{2g}-1)$ of $I$ maps to a family of elements of $\widehat{I}$,
which the Magnus expansion shows to be a $\widehat{\Q[\pi]}$-basis of~$\widehat{I}$.)
The third isomorphism flips the two factors of the tensor product and applies the antipodes to each of them.
(Note that it swaps the left $\widehat{S}$-module structure and the right $\widehat{S}$-module structure.)
Finally, the fourth isomorphism in \eqref{eq:four_isos}  is induced by the isomorphism $\theta: \widehat{I} \to \widehat{T}_{\geq 1}$.
(That it is well-defined follows from the fact that the expansion $\theta$ has group-like values.)     

We now show that the diagram \eqref{eq:psi's} is commutative for any $M\in \Hcyl$.
This diagram can be decomposed as follows:
$$
\xymatrix @!0 @R=1.3cm @C=2.5cm {
H_1\big(\Sigma,\star\, ; \widehat{S}\big) \ar[d]_-{n_{+,\star}}^-\simeq  \ar[r]^-\simeq & 
\widehat{S} \otimes_{\Q[\pi]} I \ar[d] \ar[r]^-\simeq & 
\widehat{S} \otimes_{\widehat{\Q[\pi]}}\widehat{I} \ar[r]^-\simeq \ar[d]_-\simeq  & 
\widehat{I} \otimes_{\widehat{\Q[\pi]}}  \widehat{S}  \ar[r]^-\simeq   \ar[d]_-\simeq & 
\widehat{T}_{\geq 1} \otimes_{\widehat{T}} \widehat{S} \ar[dd]_-\simeq^-{\IMagnus (\varrho^\theta(M))} \\
H_1(M ,\star\, ; \widehat{S})   \ar[r] &  \widehat{S} \otimes_{\Q[\pi']} I' \ar[r] & 
\widehat{S} \otimes_{\widehat{\Q[\pi']}}\widehat{I'} \ar[r]^-\simeq & 
\widehat{I'} \otimes_{\widehat{\Q[\pi']}}  \widehat{S}  & \\
H_1\big(\Sigma,\star\, ; \widehat{S}\big) \ar[u]^-{n_{-,\star}}_-\simeq  \ar[r]^-\simeq & 
\widehat{S} \otimes_{\Q[\pi]} I \ar[r]^-\simeq \ar[u]   & 
\widehat{S} \otimes_{\widehat{\Q[\pi]}}\widehat{I} \ar[r]^-\simeq   \ar[u]^-\simeq & 
\widehat{I} \otimes_{\widehat{\Q[\pi]}}  \widehat{S}  \ar[r]^-\simeq \ar[u]^-\simeq & 
\widehat{T}_{\geq 1} \otimes_{\widehat{T}} \widehat{S}}
$$
Here  $\pi':= \pi_1(M,\star)$ and $I'$ is the augmentation ideal of $\Q[\pi']$; the arrows of the second row are defined for the pair $(M,\star)$
as we did for the pair $(\Sigma,\star)$;  the downward vertical arrows (apart from the last one) are induced by $m_+: \Sigma \to M$,
and  the upward vertical arrows are induced by $m_-: \Sigma \to M$. 
Clearly, each cell of this diagram is commutative. 
Since the natural map $H_1(M,\star\, ; \widehat{S}) \to \widehat{S} \otimes_{\Q[\pi']} {I'}$ is surjective, 
we deduce that \emph{all} the maps of this diagram are isomorphisms. We conclude that  \eqref{eq:psi's} is commutative.

Finally, we apply Lemma \ref{lem:traces} to $\mathfrak{R} := \freeLie_{\geq 2}$ to get  the following commutative diagram: 
$$
\xymatrix{
\IAut_\omega(\widehat \freeLie) \ar[rr]^-{\IAbel} \ar[d]_-{\IMagnus}&& 
 \widehat{H_1} (\widehat{\mathfrak{h}}^+) \ar[d]^-{\ITrace}\\
 \IAut_{ \widehat{S}}\big(\widehat{T}_{\geq 1} {\otimes}_{\widehat T} \widehat{S}\big) \ar[rr]_-{\ldet } &  & \widehat{S}(H)
}
$$
By Lemma \ref{lem:usual_logdet}, we have $\det = \exp \circ\, \ldet$. 
We conclude that, for any $M\in \Hcyl$, 
\begin{eqnarray*}
\det(r(M)) & \by{eq:psi's} & \det\big(\IMagnus (\varrho^\theta(M) \big) \big) \\
& =&  \exp \big( \ldet \big(\IMagnus (\varrho^\theta(M) \big) \big) \big)\\
& =&  \exp \big( \ITrace \big(\IAbel(\varrho^\theta(M) \big) \big) \big) \ = \ \exp \big( \ITrace \big(\IAbel(M) \big) \big). 
\end{eqnarray*}

\up
\end{proof}

\subsection{The abelianization map on the Torelli group}

We now study the  group homomorphism
\begin{equation}  \label{eq:IAbel_Torelli}
\IAbel: \Torelli  \longrightarrow \widehat{H_1} (\widehat{\mathfrak{h}}^+)
\end{equation}
using, as a first step, Theorem \ref{thm:Magnus_Morita}. 
In fact, the strategy that we employed to prove it may be regarded as a generalization of Morita's approach in~\cite{morita},
when he proves that his trace map $\ITrace_k: {\mathfrak{h}_k  \to S^k(H)}$ 
vanishes on the image of the $k$-th Johnson homomorphism $\tau_k :  \mcg[k]  \to \mathfrak{h}_k$  for any   $k>1$  \cite[Theorem 6.11]{morita}. 
This vanishing phenomenon is generalized by the next proposition. 

\begin{proposition} \label{prop:ITr_IAb_Torelli}
For any $f\in \Torelli$, we have 
$$
\ITrace\big(\IAbel(f)\big) = \ITrace_1 \big(\tau_1(f)\big) = \det(f_*) 
$$
where $f_*$ denotes the automorphism induced by $f$ at the level of the twisted first homology group 
$H_1\big(\Sigma,\star\, ; \Z[H_1(\Sigma;\Z)]\big)$.
In particular, $\ITrace\big(\IAbel(f)\big) \in H_1(\Sigma;\Z) \subset \widehat{S}(H)$.
\end{proposition}

\begin{proof}
Set $H^\Z:= H_1(\Sigma;\Z)$ and denote by $R:= \Z[H^\Z]$ its group ring. 
Let $f \in \Torelli$. Then we have
$$
r(f) =  \Id_{\widehat{S}}  \otimes_{R } f_* 
\in \Aut_{\widehat{S}} \big( H_1\big(\Sigma,\star\, ; \widehat{S}\big)\big)
$$
where $r$ is the  Magnus representation as defined at the end of Section \ref{subsec:classical_Magnus},
$\widehat{S}$ is regarded as a right $R$-module via the exponential map $H^\Z \to \widehat{S}$ 
and we identify
$$
H_1\big(\Sigma,\star\, ; \widehat{S}\big) \simeq  \widehat{S} \otimes_{R } 
H_1\big(\Sigma,\star\, ;  R \big).
$$
Since $R^\times = \pm H^\Z$ and since a symplectic matrix has determinant one, 
we have $\det(f_*)=h$ for some $h\in  H^\Z$. It follows that
$$
\det\big( r(f) \big) =  \exp(h)
$$
and we deduce from Theorem \ref{thm:Magnus_Morita} that $\exp\big( \ITrace (\IAbel(f)) \big)=  \exp(h)$.
So $\ITrace (\IAbel(f)) =h$.
\end{proof}

\begin{remark}  \label{rem:ES}
Morita's result that $\ITrace_k \circ \tau_k: \mcg[k]  \to S^k(H)$ is zero (for any ${k\geq 2}$) can also be generalized in the following, different direction:
the composition ${\Trace_k \circ \tau_k}:  {\mcg[k]  \to C_k(H)}$ is zero for any $k\geq 2$.
This result has been first observed by Enomoto and Satoh \cite{es}, 
by combining Satoh's work \cite{satoh} to Hain's result that the rational graded Lie algebra associated to the Johnson filtration is generated by its degree one part \cite{hain97}.
In our setting, that ${\Trace_k \circ \tau_k=0}$ 
directly follows from Proposition \ref{prop:ext_satoh} which is logically independent from the results of \cite{satoh} and \cite{hain97}. \hfill $\blacksquare$
\end{remark}

To go further in our study of the homomorphism \eqref{eq:IAbel_Torelli}, 
we observe that it induces an $\Sp$-equivariant homomorphism 
$\IAbel: H_1(\Torelli;\Q)  \to \widehat{H_1} (\widehat{\mathfrak{h}}^+)$ 
where $H_1(\Torelli;\Q)=\Torelli/[\Torelli,\Torelli] \otimes_\Z \Q $
has the $\Sp(H)$-action induced by the conjugacy action of $\mcg$ on $\Torelli$. 
(This follows from the fact that \eqref{eq:IAbel_Torelli} is $\mcg$-equivariant using \cite[Lemma~2.2.8]{AN}.) 
For any $k\geq 1$, let
$$
\IAbel_k:  H_1(\Torelli;\Q)  \longrightarrow {H_1}({\mathfrak{h}}^+)_k
$$
be the $\Sp$-equivariant  homomorphism 
obtained from $\IAbel$ by projection onto the degree $k$ part of~$\widehat{H_1} (\widehat{\mathfrak{h}}^+)$.

\begin{proposition} \label{prop:t}
The homomorphism $\IAbel_k$ has the following~properties:
\begin{itemize}
\item[(i)] if $g=1$, $\IAbel_2$ is an isomorphism and $\IAbel_k=0$ for all $k\neq 2$;
\item[(ii)] if $g \geq 3$, $\IAbel_1$ is an isomorphism and $\IAbel_{2k}=0$ for all $k \ge 1$.
\end{itemize}
\end{proposition}

\begin{proof}
Let $T$ be the (right-handed) Dehn twist around a curve parallel to $\partial \Sigma$, and consider 
$$
t:=\IAbel (T) \in \widehat{H_1} (\widehat{\mathfrak{h}}^+).
$$ 
Since the automorphism of $\pi$ induced by $T$ is the conjugacy $x \mapsto \zeta x \zeta^{-1}$ by $\zeta:=[\partial \Sigma]$,
the automorphism  $\varrho^\theta (T)$ of $\widehat{\freeLie}$ is the 
conjugacy $u \mapsto {\exp(-\omega) u \exp(\omega) }$,
so that the derivation $\log \varrho^\theta (T)$ is given by $u\mapsto [u,\omega]$.
We deduce that $t$ is homogeneous of  degree~$2$.

Assume that $g=1$. Then $\mathfrak{h}_1  = 0$  so that ${H_1} ({\mathfrak{h}}^+)_2= \mathfrak{h}_2$.
Furthermore, it is easily seen that $\mathfrak{h}_2$ is one-dimensional (using the isomorphism \eqref{eq:eta}, for instance): 
it follows that $\IAbel_2$ is surjective.
Since  $\Torelli$ is infinite cyclic generated by $T$, this proves (i).

We now assume that $g\geq 3$. The map 
$\IAbel_1: H_1(\Torelli;\Q) \to \mathfrak{h}_1 \simeq \wedge^3 H$ is essentially $\tau_1$ 
and the latter is  an isomorphism by Johnson's result \cite{johnson_abelianization}. 
For any $j \ge 1$, the space $\mathfrak{h}_{j}$ is regarded as an $\Sp(H)$-submodule 
of $H^{\otimes (j+2)}$ in the following way: 
\begin{equation} \label{eq:h_T}
\mathfrak{h}_j \subset \Hom(H,\freeLie_{j+1}) \simeq H^\ast \otimes \freeLie_{j+1} 
\simeq H \otimes \freeLie_{j+1}
\subset H \otimes H^{\otimes( j+1)}. 
\end{equation}
Recall that finite-dimensional  $\Sp(H)$-irreducible modules are indexed by Young diagrams having no more than $g$ rows \cite[Sections 16 \& 17]{fh}.
The $\Sp(H)$-module $ H^{\otimes (2k+2)}$ only involves Young diagrams with an even number of boxes
whereas $\wedge^3 H$ has two components, namely $[1]$ and $[1,1,1]$,  with $1$ and $3$ boxes, respectively. 
Thus $\mathfrak{h}_{2k} \subset H^{\otimes (2k+2)}$ does not share 
any $\Sp(H)$-irreducible component with $H_1(\Torelli;\Q) \simeq \wedge^3 H$. 
This shows  (ii).  
\end{proof}

\begin{remark}
Assume that $g\geq 3$. 
The authors do not know whether  $\IAbel_{2k+1} (\Torelli)$ is zero for $k\geq 1$,  
except in the specific cases where the $\Sp(H)$-irreducible decomposition of $\mathfrak{h}_{2k+1}$ is known. 
For example, we see that $\IAbel_{3} (\Torelli)$ is trivial 
from an explicit description of $\mathfrak{h}_{3}$  (see Asada \& Nakamura \cite[Section 4]{AN}). 
Note that,  for any integer $d\geq 2$ satisfying $\IAbel_d(\Torelli)=0$, we obtain that  $\tau_d(\mcg[d])$ projects trivially on $H_1 (\mathfrak{h}^+)_d$:
this latter fact is known to be true without condition on  $d$ by the result of Hain that we have already mentioned  in Remark \ref{rem:ES}.

The case $g=2$ is exceptional since $\Torelli$ projects onto a free group of infinite rank 
by a result of Mess \cite{mess}:
in particular, $H_1(\Torelli;\Q)$ is infinite-dimensional. The authors have not investigated this case in detail. \hfill $\blacksquare$
\end{remark}

\section{Rational abelianization of the group of homology cobordisms} \label{sec:abelian}
 
Let $\Sigma$ be a compact connected oriented surface of genus $g\geq 1$ with one boundary component.
Recall that $\cob= \cob(\Sigma)$ is the corresponding monoid of homology cobordisms,
and  $\Hcob = \Hcob(\Sigma) $ is the  group of homology cobordisms.

\subsection{Rational abelian quotients}

We shall prove the following.

\begin{theorem}\label{thm:Q-abel}
There exists a non-trivial invariant $\widetilde{I}: \cob \to \Q$ of homology cobordisms with the following properties:
\begin{enumerate}
\item[(i)] $\widetilde{I}$  is invariant under the relation of  $4$-dimensional homology cobordism;
\item[(ii)] $\widetilde{I}$ is additive, i.e. $\widetilde{I}(M \cdot N) = \widetilde{I}(M) + \widetilde{I}(N)$ for all $M,N \in \Hcob$; 
\item[(iii)] $\widetilde{I}$ vanishes on the mapping class group $\mcg \subset \cob$;
\item[(iv)] there is a $k\geq 3$ such that $\widetilde{I}$ is a finite-type invariant of degree $k$, 
and $\widetilde{I}$ is determined by the action of $\cob$ 
on the $(k+1)$-st nilpotent quotient  of $\pi$. 
\end{enumerate}
\end{theorem}

\begin{cor}\label{cor:Q-abel}
The rational abelianization $H_1 (\Hcob;\mathbb{Q})$ of the group $\Hcob$ is non-trivial.
\end{cor}

Conditions (i) and (ii) in Theorem \ref{thm:Q-abel} show that $\widetilde{I}$ induces a non-trivial group homomorphism $\widetilde{I}: \Hcob \to \Q$, 
and  Corollary \ref{cor:Q-abel} immediately follows from that.

Note that condition (iii)  is automatic in genus $g\geq 2$ since $H_1 (\mcg;\Z)$
is known to be finite cyclic in this case;
in genus $g=1$, (iii) is equivalent to say that $\widetilde{I}$ vanishes on the Dehn twist $T$ along the boundary curve
since $H_1 (\mcg;\Z)$ is infinite cyclic generated by $T$ in this case.
(See \cite[Theorem 5.1]{korkmaz}, for instance.)
Besides, it follows from (iii) that the value of $\widetilde{I}$ on  a cobordism $M \in \cob$  only depends on the oriented $3$-manifold underlying $M$
(i.e., $\widetilde{I}$ is insensitive to the boundary parametrizations $m_\pm: \Sigma \to M$).

We now explain how to construct such an invariant $\widetilde{I}$ using the results of Section~\ref{sec:h_cob}. 
Compose the map $\Abel^\theta : \Hcob \to \widehat{H_1} (\widehat{\mathfrak{h}}^+)$ 
with the projection onto the coinvariant quotient  $\widehat{H_1} (\widehat{\mathfrak{h}}^+)_{\Sp}$. 
The resulting map 
\begin{equation}   \label{eq:canonical}
\Hcob \longrightarrow \widehat{H_1}(\widehat{\mathfrak{h}}^+ )_{\Sp}
\end{equation}
is a group homomorphism, which is independent of the choice of the symplectic expansion $\theta$ by  (\ref{eq:for_further_use}). 
Our claim is that the space  $\widehat{H_1} (\widehat{\mathfrak{h}}^+)_{\Sp}$ is non-trivial in degree $>2$. 
 Consequently, there is an integer $k\geq 3$ and  a non-zero linear form $I:  \widehat{H_1}  (\widehat{\mathfrak{h}}^+)_{\Sp} \to \Q$
 which is supported in the degree $k$ part.
 By composing \eqref{eq:canonical}  with this form $I$, we obtain a non-trivial group homomorphism $\widetilde{I}: \Hcob \to \Q$.
 That  $\widetilde{I}$ satisfies (iii) follows from the previous paragraph
 and  the fact (observed in  the proof of Proposition~\ref{prop:t} for $g=1$) that  $\IAbel(T)$ is homogenous of degree~$2$.
 That $\widetilde{I}$ satisfies (iv) directly follows from its construction, 
 Corollary \ref{cor:fti} (which shows that $\widetilde{I}$ is of degree at most $k$)
 and Remark \ref{rem:claspers} (which adds that $\widetilde{I}$ is of degree $k$, exactly). 

Thus it remains to prove the above claim about  $\widehat{H_1} (\widehat{\mathfrak{h}}^+)_{\Sp}$.
The proof is divided into two cases, $g=1$ and $g>1$,  which reflects the difference of 
the structure of the Lie algebra $\widehat{\mathfrak{h}}^+$. 
Before that, we fix some notation which will be useful for both cases. 

\subsection{Notations}

Recall that $\widehat{H_1} (\widehat{\mathfrak{h}}^+)$
 has a direct product decomposition coming from the grading.  
That is, we have
\[\widehat{H_1}(\widehat{\mathfrak{h}}^+ ) = 
\prod_{k=1}^\infty \widehat{H_1}(\widehat{\mathfrak{h}}^+ )_k\]
with 
\[\widehat{H_1}(\widehat{\mathfrak{h}}^+ )_k= H_1(\mathfrak{h}^+)_k 
= \mathfrak{h}_k \ \Big/  \!\!\!\!   \sum_{\begin{subarray}{c}i+j=k \\ i\geq 1, j\geq 1 \end{subarray}}
[\mathfrak{h}_i,\mathfrak{h}_j].\] 
By fixing a  symplectic basis $(a_1,\dots,a_g,b_1,\dots,b_g)$ of $H$, 
we identify $\Sp(H)$ with the classical group $\Sp(2g;\Q)$.
Note that the choice of this basis is just for the purpose of doing computations: 
indeed the  linear forms  $I:  \widehat{H_1}  (\widehat{\mathfrak{h}}^+)_{\Sp} \to \Q$ 
that we shall construct will not  depend on this choice. 

We now recall from \cite[Section 4]{morita_survey} how to define $\Sp$-invariant linear forms on $\mathfrak{h}$.
For any integer $k\geq 1$, the degree $k$ part  $\mathfrak{h}_k$ of $\mathfrak{h}$ 
is contained in $H^{\otimes (k+2)}$ as an $\Sp(H)$-submodule: see \eqref{eq:h_T}. 
It follows from elementary invariant theory that $\big(H^{\otimes( k+2)}\big)_{\Sp}$ is trivial if $k$ is odd:
therefore  $(\mathfrak{h}_k)_{\Sp}=0$ for $k$ odd.
Assume now that $k$ is even. For any partition $p$ of the set $\{1,\dots,k+2\}$ into $(k/2+1)$ ordered pairs, 
we can apply to  $\mathfrak{h}_k$ the linear  map 
$$
C_p: H^{\otimes( k+2)} \longrightarrow \Q
$$ 
defined by contracting the components of the tensors that are matched by $p$. 
It~also follows from elementary invariant theory that
any $\Sp$-invariant linear form on $H^{\otimes (k+2)}$ is a linear combination of such maps.

\subsection{Proof of Theorem \ref{thm:Q-abel}: the genus $1$ case}\label{subsec:genus1}

For each integer $k\geq 1$, the irreducible decomposition of  $\mathfrak{h}_k$ as a representation of $\Sp(2;\mathbb{Q})=\SL (2;\mathbb{Q})$
is obtained by using an explicit character formula (see \cite[proof of Prop. 8.2]{mss4}). 
Recall that the irreducible polynomial representations of  $\SL(2;\mathbb{Q})$  are 
given by the symmetric powers  $S^j H$ for all $j \geq 0$, which correspond to the Young diagrams~$[j]$.
For instance, $\mathbb{Q}=S^0 H =[0]$ is called the \emph{invariant representation}. 

We will use  Table \ref{tab:hg=1}, which is borrowed to \cite{mss4}. 
For instance,  the $1$-st, $3$-rd, $5$-th and $8$-th rows of this table read 
$$
\mathfrak{h}_1 =\mathfrak{h}_3=\mathfrak{h}_5=0, \quad 
\mathfrak{h}_8 =S^6 H \oplus (S^2 H)^{\oplus 2}.
$$
We will also use the following well-known formula giving the irreducible decomposition of the tensor product
of two symmetric powers (where $k \ge l$): 
\[S^k H\otimes S^l H = S^{k+l} H \oplus S^{k+l-2} H \oplus \cdots \oplus S^{k-l} H\] 
This formula implies that $(S^k H\otimes S^l H)_{\SL} \simeq \mathbb{Q}$
if and only if $k=l$. Otherwise we have $(S^k H\otimes S^l H)_{\SL}=0$, that is $S^k H\otimes S^l H$ includes no invariant representation. 

\begin{table}[h]
\caption{The irreducible decompositions of $\mathfrak{h}_k$ in genus  $g=1$  (see \cite[Table 8]{mss4}).}
\begin{center}
\begin{tabular}{|c|l|}
\noalign{\hrule height0.8pt}
\hfil $k$& $\hspace{1cm}\text{irreducible components of $\mathfrak{h}_k$}$ \\
\hline
$1$ & $\{0\}$ \\
\hline
$2$ & $[0]$ \\
\hline
$3$ & $\{0\}$ \\
\hline
$4$ & $[2]$ \\
\hline
$5$ & $\{0\}$ \\
\hline
$6$ & $[4],[0]$\\
\hline
$7$ & $[3]$ \\
\hline
$8$ & $[6], 2[2]$ \\
\hline
$9$ & $[5],[3],[1]$ \\
\hline
$10$ & $[8],[6],3[4],[2],3[0]$ \\
\hline
$11$ & $[7],2[5],4[3],2[1]$ \\
\hline
$12$ & $[10],[8],5[6],4[4],8[2]$ \\
\hline
$13$ & $2[9],3[7],8[5],9[3],6[1]$\\ 
\hline
$14$  & $[12],[10],7[8],9[6],18[4],11[2],11[0]$ \\ 
\hline
$15$ & $2[11],5[9],14[7],21[5],26[3],17[1]$\\ 
\hline
$16$ & $[14],2[12],9[10],16[8],38[6],38[4],46[2],10[0]$\\ 
\hline
$17$ & $2[13],7[11],23[9],42[7],68[5],72[3],48[1]$ \\ 
\hline
$18$ & $[16],2[14],12[12],26[10],67[8],96[6],138[4],100[2],57[0]$\\   
\noalign{\hrule height0.8pt}
\end{tabular}
\end{center}
\label{tab:hg=1}
\end{table}

We deduce from Table \ref{tab:hg=1} that $ {H_1}({\mathfrak{h}}^+)_k$ 
has $1$, $1$, $3$  copies of the invariant representation $[0]$ for $k=2,6,10$ respectively. 
This is because there are no chances to eliminate the invariant representations present in $\mathfrak{h}_k$ 
by brackets of lower degree terms. 
From Table \ref{tab:hg=1}, and  using the fact that $\left(\wedge^2 (S^{2j+1}H)\right)_{\SL} \simeq  \mathbb{Q}$ for any $j$,
it can be checked that ${H_1}({\mathfrak{h}}^+ )_{14}$ and 
${H_1}({\mathfrak{h}}^+ )_{18}$ have at least  $11-2=9$ and $57-35=22$ copies of the invariant representation $[0]$. 
(The authors have not checked whether one of the $10$ copies of $\mathfrak{h}_{16}$ 
survives in ${H_1}({\mathfrak{h}}^+ )_{16}$.)
This concludes the proof of Theorem \ref{thm:Q-abel} for $g=1$ 
by considering $k=6,10, 14$ or $18$, for instance.

In the rest of this subsection, we   specify some non-trivial  $\SL$-invariant linear forms  
$I_k: {H_1}({\mathfrak{h}}^+ )_k\to \Q$ in degrees  $k=6$ and $10$. 
Using an explicit formula for a symplectic expansion $\theta$ of genus $1$ (up to degree $k+1$),
 we could write down explicit formulas of the resulting invariants $\widetilde{I}_k: \cob \to \Q$.
As a warm up, we start by considering the case $k=2$, 
but we emphasize that the corresponding homomorphism 
$\widetilde{I}_2: \Hcob \to \Q$ is not trivial on $\mcg$.

\begin{example}\label{ex:genus1degree2}
Since $\mathfrak{h}_2$ is one-dimensional  and $\mathfrak{h}_1=\{0\}$, 
we have 
$$
\mathfrak{h}_2 =  H_1(\mathfrak{h}^+)_2 = \big(  H_1(\mathfrak{h}^+)_2  \big)_{\SL} \simeq \Q. 
$$
Let $I_2: \mathfrak{h}_2 \to \Q$ be the $\SL$-equivariant homomorphism obtained by restricting 
\[  C_{(12)(34)} : H^{\otimes 4} \longrightarrow \Q, \quad x_1 \otimes x_2 \otimes x_3 \otimes x_4 \longmapsto \omega (x_1,x_2) \omega (x_3,x_4)\]
to $\mathfrak{h}_2$. 
Let $t\in \mathfrak{h}_2$ be the derivation of $\freeLie$ defined by  $t(u):=[u,\omega]$.
Then $I_2(t)=6$ and we deduce that $I_2$ is an isomorphism. 
Note that $t$ is the value of $\IAbel$ on the Dehn twist $T$ along 
the boundary curve (see the proof of Proposition \ref{prop:t}).
Since~$T$ generates $H_1(\mcg;\Z) \simeq \Z$, 
we deduce that
$\widetilde{I}_2:\Hcob \to \Q$ induces an isomorphism 
between $H_1(\mcg;\Q)$ and $\Q$.  \hfill $\blacksquare$  
\end{example}

\begin{example}
Let $I_6: \mathfrak{h}_6 \to \mathbb{Q}$ be the $\SL$-equivariant homomorphism defined by restricting 
$C_{(12)(34)(56)(78)}: H^{\otimes 8 } \to \Q $ to $\mathfrak{h}_6$. By a computer calculation, we find that
$$
I_6\, \eta \Big(\quad \labellist
\small\hair 2pt
 \pinlabel {$a$} [r] at 1 78
 \pinlabel {$b$} [t] at 71 4
 \pinlabel {$a$} [t] at 144 4
 \pinlabel {$b$} [t] at 216 4
 \pinlabel {$a$} [t] at 288 4
 \pinlabel {$b$} [t] at 360 4
 \pinlabel {$a$} [t] at 432 4
 \pinlabel {$b$} [l] at 505 77
\endlabellist
\includegraphics[scale=0.2]{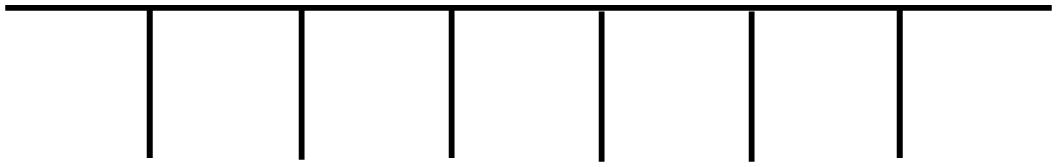} \quad \Big) =72
$$
where $a:=a_1$, $b:=b_1$
and $\eta: \A^{t,c}(H) \to \mathfrak{h}^+$ is the isomorphism recalled at~\eqref{eq:eta}. 
Hence $I_6$ induces an isomorphism $({H_1}({\mathfrak{h}}^+)_6)_{\SL} \simeq \mathbb{Q}$.  \hfill $\blacksquare$ 
\end{example}

\begin{example}
Let $I_{10}^{(1)}, I_{10}^{(2)}, I_{10}^{(3)}: \mathfrak{h}_{10} \to \Q$ be the $\SL$-equivariant homomorphisms defined by restricting
\begin{align*}
&C_{(1\, 2)(3\, 4)(5\, 6)(7\, 8)(9\,10)(11\,12)}, \\
&C_{(1\, 2)(3\, 5)(4\, 6)(7\, 8)(9\,10)(11\,12)}, \\ 
&C_{(1\, 2)(3\, 6)(4\, 7)(5\, 9)(8\,10)(11\,12)}
\end{align*}
to $\mathfrak{h}_{10}$, respectively. Consider also the following elements of $\mathfrak{h}_{10}$:
\begin{align*}
\xi_1 &:= 
\quad \eta\Big( \quad { \labellist
\small\hair 2pt
 \pinlabel {$a$} [r] at 2 77
 \pinlabel {$b$} [t] at 71 4
 \pinlabel {$a$} [t] at 145 4
 \pinlabel {$b$} [t] at 217 4
 \pinlabel {$a$} [t] at 289 4
 \pinlabel {$b$} [t] at 361 4
 \pinlabel {$a$} [t] at 432 4
 \pinlabel {$b$} [t] at 505 4
 \pinlabel {$a$} [t] at 575 4
 \pinlabel {$b$} [t] at 649 4
 \pinlabel {$a$} [t] at 721 4
 \pinlabel {$b$} [l] at 793 78
\endlabellist
\centering
\includegraphics[scale=0.2]{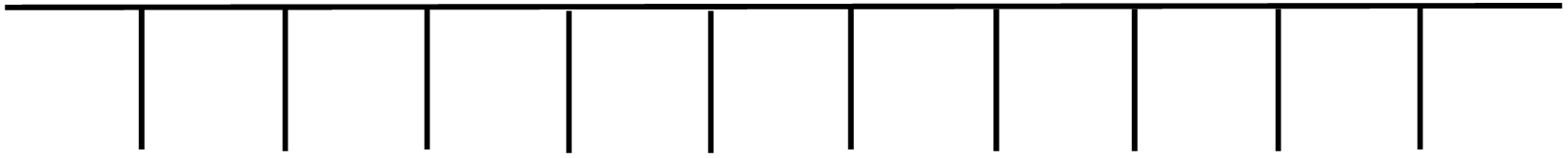}} \quad \Big) , \\[0.2cm]
\xi_2 &:= \quad \eta\Big( \quad { \labellist
\small\hair 2pt
 \pinlabel {$a$} [r] at 2 77
 \pinlabel {$b$} [t] at 71 4
 \pinlabel {$a$} [t] at 145 4
 \pinlabel {$a$} [t] at 217 4
 \pinlabel {$b$} [t] at 289 4
 \pinlabel {$a$} [t] at 361 4
 \pinlabel {$b$} [t] at 432 4
 \pinlabel {$b$} [t] at 505 4
 \pinlabel {$a$} [t] at 575 4
 \pinlabel {$b$} [t] at 649 4
 \pinlabel {$a$} [t] at 721 4
 \pinlabel {$b$} [l] at 793 78
\endlabellist
\centering
\includegraphics[scale=0.2]{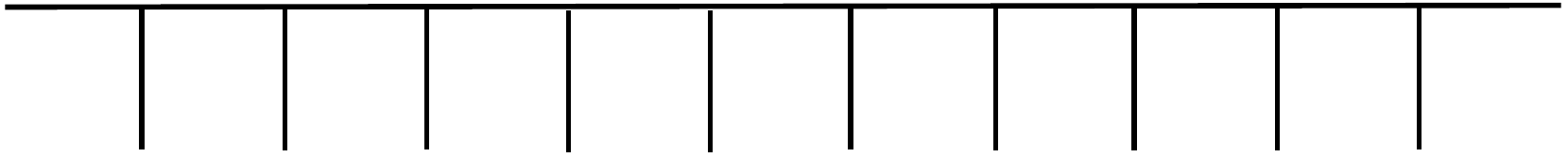}} \quad \Big) , \\[0.2cm]
\xi_3 & :=
\quad \eta\Big( \quad { \labellist
\small\hair 2pt
 \pinlabel {$a$} [r] at 2 77
 \pinlabel {$b$} [t] at 71 4
 \pinlabel {$a$} [t] at 145 4
 \pinlabel {$a$} [t] at 217 4
 \pinlabel {$a$} [t] at 289 4
 \pinlabel {$a$} [t] at 361 4
 \pinlabel {$b$} [t] at 432 4
 \pinlabel {$b$} [t] at 505 4
 \pinlabel {$b$} [t] at 575 4
 \pinlabel {$b$} [t] at 649 4
 \pinlabel {$a$} [t] at 721 4
 \pinlabel {$b$} [l] at 793 78
\endlabellist
\centering
\includegraphics[scale=0.2]{twelve_legs}} \quad \Big) .
\end{align*}
By a computer calculation, we find that 
\[\Big(I_{10}^{(j)} (\xi_i)\Big)_{i,j}=\begin{pmatrix}
324 & 32 & 18 \\
0 & -24 & -24 \\
0 & 0 & 150
\end{pmatrix}.\]
So  the triplet $\big(I_{10}^{(1)}, I_{10}^{(2)}, I_{10}^{(3)}\big)$ induces an isomorphism $({H_1}({\mathfrak{h}}^+)_{10})_{\SL} \simeq \mathbb{Q}^3$.
Thus we get three linearly independent homomorphisms $\widetilde{I}_{10}^{(1)}, \widetilde{I}_{10}^{(2)}, \widetilde{I}_{10}^{(3)}: \Hcob \to \Q$. \hfill $\blacksquare$ 
\end{example}

\subsection{Proof of Theorem \ref{thm:Q-abel}: the higher genus cases}\label{subsec:higher_genus}

In the case $g \ge 2$, we can take two approaches which are logically independent of each other. 
Note that the structure of $\widehat{\mathfrak{h}}^+$ for $g \ge 2$ is fairly different from  that  for $g=1$. 

One approach uses Kontsevich's theorem \cite{kontsevich1, kontsevich2} 
(see also Conant \& Vogtmann \cite{cv0}) to study ${H_1}({\mathfrak{h}}^+)_{\Sp}$. 
From this, we obtain an isomorphism 
\begin{equation}    \label{eq:Kontsevich}
\varinjlim_{g \to \infty} \big({H_1}({\mathfrak{h}}^+ )_{2n}\big)_{\Sp} 
\simeq H^{2n-1} \big(\Out (F_{n+1});\mathbb{Q}\big),
\end{equation}
for all $n \ge 1$, where 
$\Out (F_{n+1})$ denotes the outer automorphism group of the free group of rank $n+1$. 
When $n=6$, we see that $\big({H_1}({\mathfrak{h}}^+ )_{12}\big)_{\Sp}$ 
is stably isomorphic to $H^{11} (\Out (F_7);\mathbb{Q})$. It turns out that, very recently, 
Bartholdi \cite{bartholdi2} determined $H^\ast (\Out (F_7);\mathbb{Q})$ by using a computer 
and it results that $H^{11} (\Out (F_7);\mathbb{Q}) \simeq \mathbb{Q}$. 
Hence we see that $\big({H_1}({\mathfrak{h}}^+ )_{12}\big)_{\Sp}$ is non-trivial for sufficiently large $g$. 

To prove our claim for all $g \ge 2$, we have to take 
another approach which uses results
from a forthcoming paper \cite{mss5} by Morita, Suzuki and the second-named author.  
This paper provides an explicit $\Sp$-invariant linear form 
$$
I_{12}: \mathfrak{h}_{12} \longrightarrow \mathbb{Q}
$$ 
dual  to the above-mentioned 
stable class of $({H_1}\big({\mathfrak{h}}^+)_{12}\big)_{\Sp}$. 
To be more specific,  $I_{12}$ is given there as the restriction of a map 
$C: H^{\otimes 14} \to \Q$ to $\mathfrak{h}_{12}$ where $C$ is defined by a linear combination 
\begin{align*}
C&:=2160\, C_{(1\, 2)(3\, 9)(4\,11)(5\,12)(6\,14)(7\,13)(8\,10)}\\
&\quad -2616\, C_{(1\, 2)(3\, 9)(4\,11)(5\,13)(6\,12)(7\,14)(8\,10)}\\
&\quad -180\, C_{(1\, 2)(3\, 9)(4\,11)(5\,12)(6\,13)(7\,14)(8\,10)} \ +\cdots
\end{align*}
of 647 multiple contractions. 
It is checked in \cite{mss5} that the form $I_{12}$ vanishes on the subspace 
$\sum_{i=1}^6 [\mathfrak{h}_i, \mathfrak{h}_{12-i}]$, 
and  that it is non-trivial for all $g \ge 2$ by an explicit computation. 
This concludes the proof of Theorem \ref{thm:Q-abel} for $g\geq 2$.

Finally, we mention that $\widetilde{I}_{12}: \Hcob \to \mathbb{Q}$ is invariant under stabilization.
Recall the notations of  Corollary \ref{cor:stabilization}: 
$\Sigma'$ is a surface of genus $g+1$ in which $\Sigma$ is embedded, 
$\Hcob'$ is the corresponding group of homology cobordisms and the map $\Hcob \to \Hcob'$ is the canonical one.

\begin{proposition} \label{prop:stab_I12}
The following diagram is commutative:
$$
\xymatrix{
\Hcob \ar[d] \ar[r]^-{ \widetilde{I}_{12} }  & \Q\\
\Hcob' \ar[ru]_-{ \widetilde{I}_{12} }  & 
}
$$
\end{proposition}

\begin{proof}
Let $M \in \Hcob$. We choose an $f\in \mcg$ such that $\sigma(f)= \sigma(M)^{-1} \in \Sp(H)$: then $M f\in \Hcyl$.
We denote by   $M'\in \Hcob'$ and  $f'\in \mcg'$ the elements corresponding to $M$ and $f$, respectively, by stabilization. 
Then 
$$
\widetilde{I}_{12}(M) = \widetilde{I}_{12}(M) + \widetilde{I}_{12}(f) =  \widetilde{I}_{12}(M f) = I_{12}\big( \IAbel (M f) \big)
$$
where $I_{12}$ is regarded as a linear form on $\widehat{H_1}(\widehat{\mathfrak{h}}^+ )$.
Similarly we get $\widetilde{I}_{12}(M') =  I_{12}\big( \IAbel (M' f')\big)$. 
By Corollary \ref{cor:stabilization}, we know that $ \IAbel (M f)$ is mapped to  $ \IAbel (M' f')$ 
by the homomorphism $\widehat{H_1}(\widehat{\mathfrak{h}}^+ ) \to \widehat{H_1}(\widehat{\mathfrak{h}}^{'+} )$ corresponding to the stabilization $H \to H'$.
Besides, since $I_{12}$ is defined by some contractions with the homology intersection form $\omega$, 
the triangle
$$
\xymatrix{
\mathfrak{h}_{12}  \ar[d] \ar[r]^-{ {I}_{12} }  & \Q\\
\mathfrak{h}'_{12}  \ar[ru]_-{ {I}_{12} }  & 
}
$$
certainly commutes. We conclude that $\widetilde{I}_{12}(M)  = \widetilde{I}_{12}(M')$.

We also give another proof which works only for $g$ sufficiently large,  
but not assumes familiarity with the LMO homomorphism.   
The dimensions of   $(\mathfrak{h}_{12})_{\Sp}$ and 
$\big({H_1}({\mathfrak{h}}^+ )_{k}\big)_{\Sp}$ for $k \le 11$ stabilize for $g$ sufficiently large.
(See \cite[Theorem~1.2]{mss4} for a computation of the stable range of  $(\mathfrak{h}_{2n})_{\Sp}$ for any $n\geq 1$;
it is possible to estimate the stable range of $\big({H_1}({\mathfrak{h}}^+ )_{k}\big)_{\Sp}$ from this.) 
Furthermore we have $\big({H_1}({\mathfrak{h}}^+ )_{k}\big)_{\Sp} =0$  for $k\leq 11$ and for $g$ sufficiently large: 
for $k=2n+1$ with $n\geq 0$, this follows from the fact that $(\mathfrak{h}_{2n+1})_{\Sp}=0$;
for $k=2n$, this follows from \eqref{eq:Kontsevich} using the fact that $H^{2n-1} \big(\Out (F_{n+1});\mathbb{Q}\big)=0$  for $n\in \{1,\dots,5\}$
(see \cite{HV} for $n\in\{1,2,3\}$, \cite{gerlits} for $n=4$ and \cite{ohashi}  for $n=5$). 
Consider now the difference
$$
D:=\widetilde{I}_{12}\circ s - \widetilde{I}_{12}: \Hcob \longrightarrow \mathbb{Q}
$$ 
where $s:\Hcob \to \Hcob'$ is the canonical map. 
The restriction of $\widetilde{I}_{12}$ to $\Hcob [12]$ factors through 
the $12$th Johnson homomorphism, which is invariant under stabilization.  
Hence  $D$ vanishes on $\Hcob[12]$ and induces a map $\overline{D}: \Hcob/\Hcob[12] \to \mathbb{Q}$. 
Consider the restriction of $\overline{D}$ to $\Hcob [11]/\Hcob[12] \simeq \mathfrak{h}_{11}^\mathbb{Z}$. 
The resulting map $\overline{D}: \mathfrak{h}_{11}^\mathbb{Z} \to \mathbb{Q}$ is an 
$\Sp(H^\mathbb{Z})$-equivariant homomorphism
and so, by \cite[Lemma 2.2.8]{AN}, it induces an $\Sp(H)$-equivariant homomorphism 
$\overline{D} \otimes_\Z \mathbb{Q}: \mathfrak{h}_{11} \to \mathbb{Q}$. 
Since $I_{12}$ vanishes on commutators,
 $\overline{D} \otimes_\Z \mathbb{Q}$ factors through $H_1 ({\mathfrak{h}}^+)_{11}$ 
but, being $\Sp(H)$-equivariant, it must then factor through $\big(H_1 ({\mathfrak{h}}^+)_{11}\big)_{\Sp}=0$.
We deduce that $\overline{D}: \Hcob/\Hcob[12] \to \Q$ is trivial on $\Hcob[11]$, so that it induces a map $\overline{D}_1: \Hcob/\Hcob[11] \to \mathbb{Q}$.
Proceeding inductively, we obtain successively that $D$ vanishes on 
$$
\Hcob[12] \subset \Hcob[11] \subset \cdots \subset \Hcob[2] \subset \Hcob[1]=\Hcyl.
$$
Using again the facts that $\widetilde{I}_{12}$ is additive, that it vanishes on $\mcg$ and that $\Hcob= \Hcyl \cdot \mcg$, 
we conclude that $D:\Hcob \to \Q$ is trivial.
\end{proof}

\subsection{Final remarks}\label{subsec:finalremark}

As we saw at \eqref{eq:canonical}, 
the map $\Abel^\theta: \Hcob \to \widehat{H_1}(\widehat{\mathfrak{h}}^+ )$ induces  (for any symplectic expansion $\theta$)  a linear map 
$$
H_1(\Hcob;\Q) \longrightarrow \widehat{H_1}(\widehat{\mathfrak{h}}^+ )_{\Sp}
$$
(which does not depend on $\theta$). It is surjective after truncation at any degree,
by the last statement of Theorem \ref{thm:main}.
The authors do not know whether it is injective.

The virtual cohomological dimension of $\Out (F_{n+1})$ is $2n-1$ by Culler \& Vogtmann \cite{cuv}. 
Every time one finds a non-trivial element of the top-dimensional rational cohomology group of $\Out (F_{n+1})$,  
the first method described in Section~\ref{subsec:higher_genus} produces 
a new invariant   $\widetilde{I}_{2n}$ of homology cobordisms
satisfying the conditions of Theorem \ref{thm:Q-abel} with $k:=2n$. 
Note that, if it existed, $\widetilde{I}_{2n}$ would be invariant under stabilization of the reference surface $\Sigma$ 
(using the same arguments as in the first proof of Proposition \ref{prop:stab_I12})
 and it would exist at least for $g$ in  the stable range  of $\big({H_1}({\mathfrak{h}}^+ )_{2n}\big)_{\Sp}$.
Yet, deciding whether $H^{2n-1}(\Out (F_{n+1});\mathbb{Q})$ is non-trivial for $n>6$ 
seems to be a very difficult problem.

\appendix

\section{A noncommutative version of the log-determinant} \label{sec:cyclic_things}

In this appendix, we define a kind of  log-determinant for  automorphisms of a free module with coefficients in a noncommutative ring.

\subsection{The noncommutative trace}

We recall from \cite[Section 1.16]{karoubi} how to define the trace of endomorphisms when the ground ring is not commutative. 

Let $R$ be an algebra (over $\Q)$ and let $M$ be a right free $R$-module of finite rank $n\geq 1$.
The ($\Q$-vector) space $\Hom_R(M,R)$ is a left $R$-module in the usual way.
By the assumption on $M$, the canonical map 
$$
\tau:M\otimes_R \Hom_R(M,R) \longrightarrow \Hom_R(M,M) =\End_R(M)
$$
is a linear isomorphism (over $\Q$). Because $R$ is not assumed to be commutative, 
the evaluation map $\operatorname{ev}:M \times  \Hom_R(M,R) \to R$ does not induce a linear map  $M\otimes_R \Hom_R(M,R)\to R$ generally speaking. 
Nonetheless, it does  induce a linear map  $\operatorname{ev}: M\otimes_R \Hom_R(M,R)\to R/[R,R]$ 
where $[R,R]$ denotes the subspace of $R$ spanned by commutators $[u,v]=uv-vu$, for all $u,v \in R$. 
Thus, the \emph{trace} map is defined~by the composition
$$
 \trace:= \operatorname{ev} \circ \tau^{-1}:  \End_R(M) \longrightarrow  R/[R,R].
$$  

The trace of endomorphisms can be computed as follows.
Let $x:=(x_1, \dots, x_n)$ be an arbitrary basis of the free $R$-module  $M$. Then 
\begin{equation} \label{eq:trace_as_usual}
\trace(g) = \sum_{i=1}^n \left(\hbox{\small $i$-th coordinate of $g(x_i )$ in the basis $x$} \right)
\end{equation}
for any $g \in \End_R(M)$.
The following properties of $ \trace$ are easily checked: 
\begin{itemize}
\item[(i)] $ \trace$ is the usual trace map  $\End_R(M)  \to R$ if $R$ is commutative;
\item[(ii)] for any $g,h\in \End_R(M)$, we have $ \trace(gh)=  \trace(hg)$.
\end{itemize}

\subsection{Groups of automorphisms} \label{subsec:group_auto}

We now assume that $R$ has an augmentation $\varepsilon:R \to \Q$,
and that the $I$-adic filtration 
$$
R=I^0 \supset I^1 \supset I^2 \supset \cdots
$$
defined by the augmentation ideal $I:= \ker(\varepsilon)$ is complete.

Let $M$ be a right free $R$-module of finite rank $n\geq 1$. 
We shall equip $M$ with the $I$-adic filtration
$$
M=  M I^0 \supset  M I^1  \supset  M I^2 \supset  \cdots 
$$
which, by our assumptions, is complete. For any $R$-linear map $f:M \to M$, 
we denote by $f_\varepsilon$ the endomorphism of the vector space  $M/MI$ that is induced by $f$.

\begin{lemma}
Any $R$-linear map $f: M \to M$ such that $f_\varepsilon=\Id_{M/MI}$  is an automorphism.
\end{lemma}

\begin{proof}
By the assumption on $f$, the $R$-linear map $\phi := f - \Id_M :M \to M$ takes values in $MI$.
An induction on $k\geq 1$ shows that $\phi^k$ takes values in $MI^k$. Hence the series
$\overline f  := \sum_{k\geq 0} (-1)^k \phi^k$ defines an $R$-endomorphism of $M$ 
such that $f \overline f = \overline f f = \Id_M$.
\end{proof}

Let $\IAut_R(M)$ be the group of $R$-linear automorphisms of $M$  that induce the identity at the level of  $M/MI$.
There is a short exact sequence of groups
\begin{equation}   \label{eq:IAut}
\xymatrix{
1 \ar[r] & \IAut_R(M) \ar[r]  &  \Aut_R(M) \ar[r]^-{f\mapsto f_\varepsilon} & \Aut(M/MI) \ar[r] & 1.
}
\end{equation}
The choice of an $R$-linear isomorphism $s:(M/MI)\otimes R \to M$ induces a group homomorphism $s: \Aut(M/MI) \to  \Aut_R(M)$,
which gives a section to the short exact sequence \eqref{eq:IAut}.
For instance, the choice of a basis $x$ of $M$ defines a linear isomorphism between $M/MI$ and $(R/I)^n\simeq \Q^n$,
which induces an $R$-linear isomorphism $s= s_x$ between $(M/MI)\otimes R \simeq \Q^n \otimes R \simeq R^n$ and $R^n \simeq M$.

The following lemma, where $\Hom_R(M,MI)$ is the ideal of $\End_R(M)$ consisting of $R$-linear maps $M \to MI$, is easily proved.

\begin{lemma}
There is a bijection $\IAut_R(M) \stackrel{\simeq }{\longrightarrow} \Hom_R(M,MI)$
defined by the logarithmic series 
$$
 \log(f) =  -\sum_{k=1}^\infty  \frac{( \Id_M -f)^k}{k}, \quad \hbox{for } f \in \IAut_R(M), 
$$
whose inverse is given by the exponential series
$$
\exp(g) =   \sum_{k= 0}^\infty  \frac{g^k}{k!}, \quad \hbox{for }   g \in \Hom_R(M,MI). 
$$
\end{lemma}

\subsection{The noncommutative  log-determinant}

We consider the same data as in Section~\ref{subsec:group_auto}:
 $(R,\varepsilon)$ is an augmented algebra whose $I$-adic filtration is complete, and  $M$ is a right free {$R$-module} of finite rank. 

We fix an $R$-linear isomorphism $s:(M/MI)\otimes R \to M$, which induces a section  $s: \Aut(M/MI) \to \Aut_R(M)$  of \eqref{eq:IAut}. 
For all $f\in \Aut_R(M)$, we set
$$
\ldet^s(f) := \trace\big( \log\big( f \circ s(f_\varepsilon)^{-1}\big) \big) \in R/[R,R]
$$
where $[R,R]$ denotes the closed subspace of $R$ spanned by commutators. 

\begin{lemma} \label{lem:prop_ldet}
The map $\ldet^s\!:\!  \Aut_R(M) \to R/[R,R]$ has the following properties:
\begin{itemize}
\item[(i)] it is a group homomorphism;
\item[(ii)] it takes values in $I/[R,R]$;
\item[(iii)] if $M= N \oplus N'$ is the direct sum of two free right $R$-modules, 
then  ${\ldet^s(g\oplus g')}= \ldet^s(g) + \ldet^s(g')$ for any $g\in \Aut_R(N)$, $g'\in \Aut_R(N')$;
\item[(iv)] for all $r\in R^\times$, we have $\ldet^s(\Id_M\!\cdot r) = \dim(M)\, \log( r/\varepsilon(r))$.
\end{itemize}
\end{lemma}

\begin{proof}
Properties (ii), (iii) and (iv) are easily checked.
We only prove (i): that
$\ldet^s(f \circ h) = \ldet^s(f) + \ldet^s(h)$ for any $f,h\in  \Aut_R(M)$. 
We first assume that $f,h\in  \IAut_R(M)$.
Then, by the BCH formula, we have
$$
\log(fh) = \log(f) +\log(h) + \frac{1}{2} \big[\log(f),\log(h)\big] + \cdots
$$
where $\big[\log(f),\log(h)\big]$ denotes the commutator $$\log(f) \log(h)-\log(h)\log(f) \in \Hom_R(M,MI^2)$$ in the algebra $\End_R(M)$
and the remaining terms are higher-length iterated commutators of $\log(f)$ and $\log(h)$.
Since $ \trace$ vanishes on commutators of $\End_R(M)$ and is filtration-preserving, we deduce that
$$
 \trace\log(fh) =  \trace\log(f) + \trace  \log(h) \ \in R/[R,R]
$$
as desired. Consider now some arbitrary elements $f,h\in \Aut_R(M)$. Then 
\begin{eqnarray*}
&& \trace \log\big( fh \,  s(f_\varepsilon h_\varepsilon)^{-1}\big) \\
&=& \trace \log\big( \big(f s(f_\varepsilon)^{-1}\big) \circ  \big( s(f_\varepsilon)\, h\,  s(h_\varepsilon)^{-1}\, s(f_\varepsilon)^{-1} \big)  \big) \\
&=& \trace \log\big(  f s(f_\varepsilon)^{-1}\big) + \trace \log \big( s(f_\varepsilon)\, h\,  s(h_\varepsilon)^{-1}\, s(f_\varepsilon)^{-1} \big) \\
&=& \trace \log\big(  f s(f_\varepsilon)^{-1}\big) + \trace\big( s(f_\varepsilon)\circ  \log \big(  h\,  s(h_\varepsilon)^{-1}\big) \circ s(f_\varepsilon)^{-1} \big)    \\
&=& \trace \log\big(  f s(f_\varepsilon)^{-1}\big) + \trace   \log \big(  h\,  s(h_\varepsilon)^{-1}\big),
\end{eqnarray*}
where the second equality follows from the previous case.
\end{proof}

\begin{example} \label{ex:ldet_matrices}
Assume that $M=R^n$. Then $M/MI=\Q^n$ and there is an obvious choice of $s$, which we will omit from the notation.
Thus we obtain a map
\begin{equation}
\label{eq:ldet_matrices}
\ldet: \GL(n;R) \longrightarrow R/[R,R]
\end{equation}
defined by 
$$
\ldet(P) :=  -  \sum_{k=1}^\infty  \frac{\trace\big( (I_n- P\, \varepsilon(P)^{-1})^k \big)}{k} 
$$
for any $P\in \GL(n;R)$, where $\varepsilon(P) \in \GL(n;\Q)$ is obtained from $P$ by applying $\varepsilon$ to all entries  
and ``$\trace$'' denotes here the usual trace of matrices.
(See \cite[Definition~3]{kricker} for a similar definition when $R$ is a noncommutative algebra of  formal power series.)

Let $K_1(R)$ be the abelianization of the infinite linear group $\GL(R)$.
By Lemma~\ref{lem:prop_ldet}, the maps \eqref{eq:ldet_matrices} defined for all $n\geq 1$ induce a group homomorphism
$$
\ldet: K_1(R) \longrightarrow R/[R,R]
$$
such that $\ldet(x)= \log(x/\varepsilon(x))$ for any $x\in R^\times$. \hfill $\blacksquare$
\end{example}

Clearly the definition of $\ldet^s(f)$ does not depend on $s$ if $f\in \IAut_R(M)$.
We define the \emph{log-determinant} map to be the composition
$$
\xymatrix{
 \IAut_R(M) \ar[r]^-\log  \ar@/_1.5pc/@{-->}[rr]_-{\ldet} & \Hom_R(M,MI ) \subset \End_R(M)  \ar[r]^-{ \trace} & R/[R,R].
}
$$
We conclude this appendix by observing that this map coincides with the usual log-determinant in the commutative case.

\begin{lemma} \label{lem:usual_logdet}
Assume that  $R$ is a \emph{commutative} algebra of  formal power series. 
Then the usual determinant $\det(f)$ of any $f  \in \IAut_R(M)$ is equal to 
$$
\exp\big( \ldet (f)\big)= \sum_{k= 0}^\infty   \frac{\ldet(f)^k}{k!}  \ \in R.
$$
\end{lemma}

\begin{proof}
Recall from Lemma \ref{lem:prop_ldet}  that  $\ldet(f)$ belongs to $I$: 
hence the exponential of  $\ldet(f)$ does converge. Thus the lemma only claims that
$$
\exp  \trace \log(f) = \det(f)
$$
for any $f  \in \IAut_R(M)$.
This is known as the ``Jacobi formula'', which is for instance proved in \cite[Section 1.1.10]{gj}.
\end{proof}

%
%	References 
%

\bibliographystyle{amsalpha}

\bibliography{Morita_trace}

\end{document}